\documentclass{amsart}
\usepackage{graphicx} 
\usepackage{amsmath}
\usepackage{float,,mathrsfs}
\usepackage{booktabs,tabularx}
\usepackage{siunitx, soul}
\usepackage{geometry}
\usepackage[thicklines]{cancel}
\usepackage{subcaption,amssymb}

\newcommand{\p}{\partial}
\newcommand{\e}{\varepsilon}
\newcommand{\ep}{\epsilon}
\newcommand{\C}{{\mathbb C}}
\newcommand{\R}{{\mathbb R}}

\newcommand{\Z}{{\mathbb Z}}

\newcommand{\cC}{{\mathcal C}}
\newcommand{\cF}{{\mathcal F}}

\newcommand{\cX}{{\mathcal X}}

\newcommand{\cO}{{\mathcal O}}
\newcommand{\fu}{{\mathbf u}}

\newcommand{\fx}{{\mathbf x}}

\newcommand{\bh}{{\bar{h}}}

\newcommand{\rw}{{\mathring{w}}}
\newcommand{\ru}{{\mathring{u}}}
\newcommand{\rv}{{\mathring{v}}}
\newcommand{\rh}{{\mathring{h}}}

\newcommand{\rA}{{\mathring{A}}}

\newcommand{\sgn}{{\rm sgn}}
\newcommand{\norm}[1]{{\big\|#1\big\|}}

\newtheorem{thm}{Theorem}
\newtheorem{lem}{Lemma}
\newtheorem{prop}{Proposition}

\newtheorem{rem}{Remark}

\numberwithin{equation}{section}
\numberwithin{thm}{section}
\numberwithin{lem}{section}
\numberwithin{prop}{section}
\numberwithin{cor}{section}
\numberwithin{rem}{section}
\numberwithin{assump}{section}
\numberwithin{definition}{section}
\allowdisplaybreaks[3]

\title[Steady triple-deck equations]{On the classical solution for the steady triple-deck equations}

\author[M. Dong]{Ming Dong}
\address{Institute of Mechanics, Chinese Academy of Sciences}
\email{dongming@imech.ac.cn} 

\author[C. Wang]{Chao Wang}
\address{School of Mathematical Sciences, Peking University}
\email{wangchao@math.pku.edu.cn} 

\author[Q. Wu]{Qin Wu}
\address{School of Mathematical Sciences, Peking University}
\email{q.wu@stu.pku.edu.cn} 

\author[Z. Zhang]{Zhifei Zhang}
\address{School of Mathematical Sciences, Peking University}
\email{zfzhang@math.pku.edu.cn} 
\date{\today}

\begin{document}

	\begin{abstract}
		
		This paper establishes the existence and uniqueness of classical solutions to the steady Triple-Deck equations, which describe incompressible boundary layer flow over localized roughness at high Reynolds numbers.  
		The triple-deck theory was developed to overcome the Goldstein singularity in classical Prandtl boundary layer theory, capturing the interaction between the viscous sublayer, main layer, and upper layer when surface roughness of height $O({\rm Re}^{-\frac58})$  is present.  The key ingredients of this paper include: (1) A decomposition separating roughness effects from nonlinear difficulties; 
		(2) A novel Green's function using Airy functions that overcomes low-frequency singularities via the non-vanishing property of $\sqrt{3}Ai(z)+Bi(z)$; 
		(3)  The introduction of weighted Sobolev norms $\norm{|\p_x|^{\frac{1}{18}}y^{\frac16}\omega}_{L^2}$ of the vorticity  yielding $M$-independent estimates for displacement  $A$.
		As a byproduct, local uniqueness of Couette flow is established when $F=0$.      
	\end{abstract}
	\maketitle

	\section{Introduction}

	\subsection{The Triple-Deck equations}
	
	In 1904, Prandtl \cite{Prandtl1904} proposed the boundary layer theory to address attached flow problems, achieving remarkable success in fluid dynamics research. However, when dealing with flow separation induced by adverse pressure gradients, this theory encounters singularity issues. Notably, Landau and Lifshitz \cite{LL1944} revealed that the velocity component normal to the body surface grows unboundedly (termed {\it the Goldstein singularity}), which hinders the solution from being further extended to the reverse-flow region \cite{Goldstein1948}. To overcome these difficulties, scholars including Messiter \cite{Messiter1970} and Stewartson \cite{Stewartson1969} developed the triple-deck theory from the late 1960s to the early 1970s. In this theory, three distinct layers—a viscous sublayer, a main layer, and an upper layer—appear in the transverse direction, for which uniformly valid approximate solutions are derived by the asymptotic matched method. 
	
	The 2D steady incompressible Navier-Stokes equations take the following dimensionless form:
	\begin{equation}\label{Incompressible NS0}
		\left\{\begin{aligned}
			&\fu_{NS}\cdot\nabla\fu_{NS} - \ep^8\Delta\fu_{NS} + \nabla p_{NS} = 0,\\
			&\nabla\cdot\fu_{NS} = 0.
		\end{aligned}\right.
	\end{equation}
	Here $(u_{NS},v_{NS})$ denotes the velocity vector, $p_{NS}$ pressure, and $\ep = {\rm Re}^{\frac18}$ with the Reynolds number ${\rm Re}\gg1$. In physics, the triple-deck structure emerges when localized surface roughness of height $\cO(\ep^5)$ is present. We then denote it as 
	$$
	\Big \{(X,Y)\big|\ Y = \ep^5F\big(\frac{X-1}{\ep^3}\big),\ X\in\R\Big\}
	$$
	for some given function $F=\cO(1)$ as $\ep\to0$. Consequently, the roughness-induced mean-flow distortion exhibits three distinct layers in the transverse direction, spanning in the $\cO(\ep^3)$ vicinity of the roughness (see figure \ref{fig0-3}).
	\begin{figure}[h]
		\centering
		\includegraphics[scale=0.35]{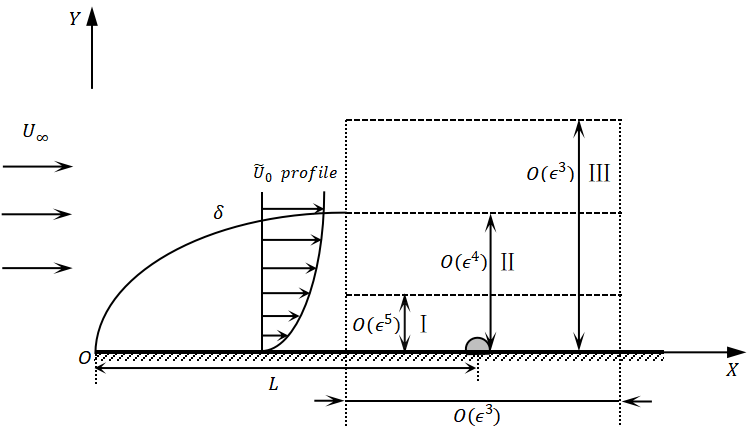}
		\caption{Demonstration of the triple-deck structure (not to scale): I   Viscous Sublayer; II  Main Layer;  III Upper Layer; $\delta$ Thickness of Prandtl layer.}
		\label{fig0-3}
	\end{figure}
	We demonstrate the scales of $(u_{NS},v_{NS},p_{NS})$ in Table \ref{tab1}.
	\begin{table}[h]
		\centering
		\begin{tabular}{|c|c|c|c|}
			\hline
			& I & II & III\\
			\hline
			$u_{NS}$ & $\cO(\ep)$ & $\tilde{U}_0(\frac{Y}{\ep^{4}}) + \cO(\ep)$ & $1 + \cO(\ep^2)$\\
			\hline
			$v_{NS}$ & $\cO(\ep^3)$ & $\cO(\ep^2)$ & $\cO(\ep^2)$\\
			\hline
			$p_{NS}$ & $\cO(\ep^2)$ & $\cO(\ep^2)$ & $\cO(\ep^2)$\\
			\hline
		\end{tabular}
		\caption{Scales of velocities and pressure in the triple-deck structure: $\tilde{U}_0$ Streamwise velocity in Prandtl layer.}
		\label{tab1}
	\end{table}
	\par Based on the scaling relations in Figure \ref{fig0-3} and Table \ref{tab1}, and via the method of matched asymptotic expansions applied to \eqref{Incompressible NS0} (see Appendix for details), we derive the governing equations for the viscous sublayer (region I) as
	\begin{equation}\label{Initial Eq0-1}
		\left\{\begin{aligned}
			&u_*\p_{x_*}u_* + v_*\p_{y_*}u_* + \frac{dP_*}{dx_*} - \p_{y_*}^2u_* = 0,\\
			&\p_{x_*}u_* + \p_{y_*}v_* = 0.
		\end{aligned}\right.
	\end{equation}
	Here $y_* = \frac{Y}{\ep^5}$ and $(\ep u_*,\ep^3 v_*, \ep^2P_*)$ are the leading-order terms of $(u_{NS}, v_{NS},p_{NS})$. The velocity field is subject to the no-slip, non-penetration conditions,
	\begin{equation}\label{BC0-1}
		(u_*,v_*)\big|_{y_* = F(x_*)} = 0.
	\end{equation}
	In the upstream limit, the velocity approaches the unperturbed Blasius profile,
	\begin{equation}\label{BC1-1}
		(u_* - y_*)\big|_{x_*\downarrow-\infty} = 0,\quad P_*\big|_{x_*\downarrow-\infty} = 0.
	\end{equation}
	To close the system, the upper boundary condition of the sublayer equations is enforced by matching with the main layer solution (see Appendix for details). Then we can derive the leading-order pressure as follows
	\begin{equation}\label{B.C. for viscous layer bounded0-0}
		P_* = -\frac{1}{\pi} \text{P.V.}\int_{\R}\frac{A_*'(\zeta)}{\zeta-x_*}d\zeta.
	\end{equation}
	And the upper boundary condition of $u_*$ is
	\begin{equation}\label{B.C. for viscous layer bounded0-1}
		( u_{*}(x_*,y_*) - y_*)\big|_{y_*=f(x_*)} = A_*(x_*).
	\end{equation}
	Here $f$ is given (to be determined later). Consequently, 
	by \eqref{Initial Eq0-1}-\eqref{B.C. for viscous layer bounded0-1}, we derive the closed system for the viscous sublayer as
	\begin{equation}\label{Initial Eq}
		\left\{\begin{aligned}
			&u_*\p_{x_*}u_* + v_*\p_{y_*}u_* + \frac{dP_*}{dx_*} - \p_{y_*}^2u_* = 0,\quad F(x_*)<y_*<f(x_*), x_*\in\R,\\
			&\p_{x_*}u_* + \p_{y_*}v_* = 0,\quad F(x_*)<y_*<f(x_*), x_*\in\R,\\
			&(u_*,v_*)\big|_{y_* = F(x_*)} = 0,\quad (u_{*}(x_*,y_*) - y_*)\big|_{y_*=f(x_*)} = A_*(x_*),\quad x_*\in\R,\\
			&(u_* - y_*)\big|_{x_*\downarrow-\infty} = 0,\quad P_*\big|_{x_*\downarrow-\infty} = 0,\quad F(-\infty)<y_*<f(-\infty),\\
			&P_* = -\frac{1}{\pi} \text{P.V.}\int_{\R}\frac{A_*'(\zeta)}{\zeta-x_*}d\zeta.
		\end{aligned}\right.
	\end{equation}
	The system  \eqref{Initial Eq} is referred to as the Triple-Deck equations.
	
	Recall that the Prandtl equations of classical boundary layers take
	\begin{equation}\label{Prandtl eq}
		\left\{\begin{aligned}
			&u_*\p_{x_*}u_* + v_*\p_{y_*}u_* + \frac{dP_*}{dx_*} - \p_{y_*}^2u_* = 0,\quad \text{in }(0,L_0)\times\R_+,\\
			&\p_{x_*}u_* + \p_{y_*}v_* = 0,\quad \text{in }(0,L_0)\times\R_+,\\
			&(u_*,v_*)\big|_{y_* = 0} = 0,\quad x_*\in\R_+,\\
			&u_{*}\big|_{x_*=0}=u_0,\quad y_*\in\R_+,
		\end{aligned}\right.
	\end{equation}
	where $u_0$ is prescribed. Mathematically, the Triple-Deck equations share a structural resemblance to the Prandtl equations \eqref{Prandtl eq}, yet they differ fundamentally in domain and pressure determination. The Prandtl equations are defined over $(0, L_0) \times \mathbb{R}_+$, with pressure $P_*$ dictated by external potential flow $u_E(x_*)$ via the Bernoulli relation 
	\begin{align}\label{P-P}
		P_* + \frac{1}{2} u_E^2=\text{Const,}
	\end{align}
	which means that the pressure is a given function.  While, the pressure of Triple-Deck model is given by \eqref{B.C. for viscous layer bounded0-0} with the unknown function $A_*$.
	
	The existence and uniqueness of classical solutions to the Prandtl equations were first established by Oleinik \cite{Oleinik1963}; subsequent decades have yielded rich results on well-posedness, asymptotic behavior, and regularity \cite{GI2021,Iyer2020,MS1985,SG2016,Serrin1967,WZ2021,WZ2023}.  Moreover, flow separation and reverse-flow phenomena have attracted considerable attention from many scholars \cite{DM2019,IM2022,MS1984,SWZ2021}.
	
	When $F=0$, it is easy to see that $(u_*,v_*,P_*,A_*)=(y,0,0,0)$ is a trivial solution to \eqref{Initial Eq}. In a recent work \cite{IM2024}, Iyer and Maekawa  proved the rigidity of this trivial solution under certain assumptions on the solution.  However, physicists found that the triple-deck structure only appears when $F\not\equiv0$. Since the $A_*$ remains unknown in this case, finding a non-trivial solution becomes a highly non-trivial problem. Numerous results are available on physical analysis and numerical simulations; see \cite{BD1988,El2010,KZR2005,Smith1973,WD2016}. Yet, if $F\not\equiv0$, the well-posedness of \eqref{Initial Eq} still lacks rigorous mathematical justification. For the unsteady case, Iyer and Vicol \cite{IV2021} and Gérard-Varet et al.\ \cite{GIM2023} established well-posedness in real-analytic spaces, whereas Dietert and Gérard-Varet \cite{DG2022} exhibited instability for some smooth initial data, indicating that the analytic class is optimal for well-posedness.    
	
	\subsection{Main results} 
	To simplify the system \eqref{Initial Eq}, we introduce the new coordinate $(x, y)$ and new unknowns $(u, v, P, A)$: 
	\begin{equation}\label{Prandtl transformation}
		\begin{aligned}
			&x=x_*,\ y = y_* - F(x_*),\\
			&(u_*,v_*,P_*,A_*) = (y + u, v + (y + u)\frac{dF}{dx}, P, A).
		\end{aligned}
	\end{equation}
	Then the problem \eqref{Initial Eq} can be reduced to
	\begin{equation}\label{Eqs viscous1}
		\left\{\begin{aligned}
			&y\p_xu+v + \p_x|\p_x|A - \p_y^2u = - u\p_xu - v\p_yu,\quad (x,y)\in\R\times[0,M],\\
			&\p_xu + \p_yv = 0,\quad (x,y)\in\R\times[0,M],\\
			&(u,v)\big|_{y=0}=(0,0),\quad u\big|_{y=f(x)-F(x)} =  A(x) + F(x),\quad x\in\R,\\
			&\lim_{x\to-\infty}u = 0,\quad 0<y<f(-\infty)-F(-\infty).
		\end{aligned}\right.
	\end{equation} 
	
	For convenience, we take $f(x) = M + F(x)$ for $M>0$. 
	We introduce the vorticity $\omega = \p_yu$ and
	\begin{equation}
		u(x,y) = I_{y}[\omega]:= \int_0^{y}\omega dy'.
	\end{equation}
	Differentiating the first equation of \eqref{Eqs viscous1} in the variable $y$, we obtain the perturbation system about $(\omega,A)$ as follows
	\begin{equation}\label{Perturbation system0}
		\left\{\begin{aligned}
			&y\p_x\omega - \p_y^2\omega = -u\p_x\omega - v\p_y\omega,\quad (x,y)\in\R\times[0,M],\\
			&u = I_{y}[\omega],\quad v = -\p_xI_{y}[u],\quad (x,y)\in\R\times[0,M],\\
			&\p_y\omega\big|_{y=0} = \p_x|\p_x|A,\quad I_{M}[\omega] = A(x) + F(x),\quad x\in\R.
		\end{aligned}\right.
	\end{equation}
	Here $F$ is a prescribed function. We denote $\Omega_M=\R\times[0,M]$. It is easy to see that the Couette flow $(u_*,v_*,P_*,A_*)$ corresponds to the trivial solution $(u,v,P,A)=(0,0,0,0)$ when $F=0$. 
	
	Now we introduce the functional framework of classical solutions. For the vorticity $\omega$, we define the weighted Sobolev spaces  by
	\begin{equation}\nonumber
		\begin{aligned}
			X_{\bar{\alpha}} =&\ \bigg\{f\big|\ \norm{f}_{X_{\bar{\alpha}}} :=  \norm{f}_{L^2(\Omega_M)} + \norm{|\p_x|^{\frac{1}{18}}y^{\frac{1}{6}}f}_{L^2(\Omega_M)} + \norm{f}_{Y_{\bar{\alpha}}}<\infty\bigg\},\\
			Y_{\bar{\alpha}} =&\ \bigg\{f\big|\ \norm{f}_{Y_{\bar{\alpha}}}:= \norm{y|\p_x|^{\frac{\bar{\alpha}}{2}}\p_xf}_{L^2(\Omega_M)} + \norm{|\p_x|^{\frac{2}{3} + \frac{\bar{\alpha}}{2}}f}_{L^2(\Omega_M)}\\
			&\qquad\qquad\qquad\quad  + \norm{|\p_x|^{\frac{1}{3}+\frac{\bar{\alpha}}{2}}\p_yf}_{L^2(\Omega_M)} + \norm{|\p_x|^{\frac{\bar{\alpha}}{2}}\p_y^2f}_{L^2(\Omega_M)}\\
			&\qquad\qquad\qquad\quad  + \norm{|\p_x|^{\frac{1+\bar{\alpha}}{2}}y^{\frac{1}{2}}\p_yf}_{L^2(\Omega_M)}<\infty\bigg\},
		\end{aligned}
	\end{equation}
	where the parameter $\bar\alpha$ is nonnegative and \(|\partial_x|^s := \mathcal{F}^{-1}[|\xi|^s \mathcal{F}[\cdot]]\) denotes the Fourier multiplier of order $s$.
	For the displacement function $A$, we define
	\begin{equation}\nonumber
		\begin{aligned}
			X_{\alpha,\infty}=\bigg\{f\big|\ \norm{|\p_x|^{\frac{5}{6}}f}_{H^{\frac{4}{3}+\frac{\alpha}{2}}(\R)} + \norm{f}_{L^{\infty}(\R)}<\infty\bigg\}.
		\end{aligned}
	\end{equation}
	Then  the functional space for \((\omega, A)\) incorporating the upper boundary condition is defined by
	\begin{equation}\nonumber
		\begin{aligned}
			\cX_{\alpha,M} = \bigg\{(\omega,A)\big|&\ \norm{(\omega,A)}_{\cX_{\alpha,M}} = \norm{\omega}_{X_0\cap X_{\alpha}} + \norm{A}_{ X_{\alpha,\infty}}<\infty,\\
			&\ \ \omega(\cdot,y)\big|_{y=M}=\frac{Ai(M\p_x^{\frac{1}{3}})}{Ai'(0)}\p_x^{\frac{2}{3}}|\p_x|A\bigg\}.
		\end{aligned}
	\end{equation}
	Here $Ai(z)$ is the Airy function satisfying $\frac{d^2}{dz^2}Ai(z) - zAi(z)=0,\ Ai\big|_{z\uparrow+\infty} = 0$.

	To seek the solution to the system \eqref{Perturbation system0}, we introduce  the following crucial decomposition $(\omega,A) = (\omega_0+\bar{\omega},A_0+\bar{A})$,
	where $(\omega_0, A_0)$  and $(\bar{\omega}, \bar{A})$ solve the following system
	\begin{equation}\label{Eq for omega0}
		\left\{\begin{aligned}
			&y\p_x\omega_0 - \p_y^2\omega_0 = 0,\quad (x,y)\in\Omega_M,\\
			&\p_y\omega_0\big|_{y=0} = \p_x|\p_x|A_0,\quad I_M[\omega_0] = A_0 + F,\quad x\in\R,
		\end{aligned}\right.
	\end{equation}
	and
	\begin{equation}\label{Eq for omega1}
		\left\{\begin{aligned}
			&y\p_x\bar{\omega} - \p_y^2\bar{\omega} = -(u\p_x\omega + v\p_y\omega)\triangleq \bh,\quad (x,y)\in\Omega_M,\\
			&u = I_y[\omega],\quad v = -\p_xI_y[u],\quad (x,y)\in\Omega_M,\\
			&\p_y\bar{\omega}\big|_{y=0} = \p_x|\p_x|\bar{A},\quad I_M[\bar{\omega}] = \bar{A},\quad x\in\R.
		\end{aligned}\right.
	\end{equation}
	
	Our first main result is the solvability of the linear system \eqref{Eq for omega0}.
	
	\begin{thm}\label{Main thm0-0}
		Let $\alpha\in(0,\frac{7}{3}]$. The linear system \eqref{Eq for omega0}  admits a solution $(\omega_0, A_0)\in\cX_{\alpha,M}$ satisfying the estimate
		\begin{equation}
			\begin{aligned}
				\norm{\omega_0}_{X_0\cap X_{\alpha}} + \norm{A_0}_{X_{\alpha,\infty}}\le&\ C\norm{F}_{H^2(\R)},
			\end{aligned}
		\end{equation}
		where $C>0$ is independent of $\alpha,M$.
	\end{thm}

	The second main result of this paper is the existence and uniqueness of classical solution to the nonlinear system \eqref{Perturbation system0}.
	
	\begin{thm}\label{Main thm0}
		Let $\alpha\in(\frac{2}{3},\frac{7}{3}]$ and $\e\in(0,\frac{1}{6}]$.  There exist a small constant $\delta_0=\delta_0(\alpha,\e)>0$ and a sufficiently large constant $M_0$, such that when  $M\ge M_0$ and  
		\begin{align*}
			\norm{F}_{H^2(\R)}\le \delta_0M^{-\frac{5+3\e}{2}},
		\end{align*}
		the system \eqref{Perturbation system0} admits a solution $(\omega,A) = (\omega_0+\bar{\omega},A_0+\bar{A})\in\cX_{\alpha,M}$ with       
		\begin{equation}\nonumber
			\begin{aligned}
				\norm{\bar{\omega}}_{X_0\cap X_{\alpha}} + \norm{\bar{A}}_{X_{\alpha,\infty}}\le&\ CM^{\frac{5+3\e}{2}}\norm{F}_{H^2(\R)}^2.
			\end{aligned}
		\end{equation}
		Here $C>0$ is independent of $\alpha,\e, M$. Moreover, $(\omega, A)$ is the unique solution to the system \eqref{Perturbation system0} within $\cX_{\alpha,M}$ if
		$$
		\norm{(\omega,A)}_{\cX_{\alpha,M}} + \norm{F}_{H^2(\R)}\le 2\delta_0M^{-\frac{5+3\e}{2}}.
		$$
	\end{thm}

	\begin{rem}
		Our theoretical results align closely with the physical phenomena. The linear component $(\omega_0,A_0)$ captures the dominant behavior of \((\omega, A)\). Considering the scaling relation $M=\cO(\text{Re}^{\frac{\alpha_0}{8}})$ for $\alpha_0\in(0,\frac{3}{4})$ (see Appendix for details), our analysis shows that the allowable roughness scale \(\cO(\text{Re}^{-a_*})\) with \(a_*\in(0,\infty)\) falls within the physically reasonable range $[0,\infty)$ (see \cite{Smith1973}). 
	\end{rem}

	\begin{rem}
		If \(F = 0\), Iyer and Maekawa \cite{IM2024} derived the local rigidity of Couette flow. Specifically, they proved that an $L^2$ strong solution $(\omega,A)$ of \eqref{Perturbation system0} is trivial if $\norm{(\omega,A)}_{sc}$ is sufficiently small, where $\norm{\cdot}_{sc}$ is a scaling invariant norm. As a byproduct of Theorem \ref{Main thm0}, we further establish the local uniqueness of Couette flow: the system \eqref{Perturbation system0} has a locally unique solution in $\cX_{\alpha,M}$, which is a trivial solution.  
		
	\end{rem}

	\begin{rem}
		We adopt $M\gg1$ in our model for consistency with physical observations. Indeed, for general $M>0$, the results may also be got by Remark \ref{rmk: M}.

	\end{rem}


	\subsection{Main difficulties and  Key ideas}

	The system \eqref{Perturbation system0} is a degenerate nonlinear parabolic equation with variable coefficients and its degeneration occurs on the wall \(y = 0\). 
	To solve the system  \eqref{Perturbation system0}, we first introduce a crucial decomposition  \eqref{Eq for omega0}  and  \eqref{Eq for omega1}, 
	which separates the confluence caused by the roughness $F$ from the mathematical difficulties arising from nonlinear terms.
	
	For the homogeneous problem \eqref{Eq for omega0}, the key ingredient is to solve the following linear problem
	\begin{equation}\label{Perturbation system0-1}
		\left\{\begin{aligned}
			&y\p_xf_e - \p_y^2f_e = h,\quad (x,y)\in\Omega_M,\\
			&\p_yf_e\big|_{y=0} = 0,\quad x\in\R.
		\end{aligned}\right.
	\end{equation}
	Taking the Fourier transformation in $x$, the linear operator of the first equation in \eqref{Perturbation system0-1} is expressed as \(i\xi y - \partial_y^2\). It shares a structural similarity with the classical Airy operator \(\frac{d^2}{dz^2} - z\). Indeed,  there exists a solution $\hat{f}_e$ to the following linear problem
	$$\left\{\begin{aligned}
		&(i\xi y -\p_y^2)\hat{f} = \hat{h},\quad (\xi,y)\in\Omega_M,\\
		&\p_y\hat{f}\big|_{y=0} = 0,\quad \xi\in\R.
	\end{aligned}\right.$$
	Similar to the proof in \cite{IM2024}, the solution satisfies
	\begin{equation}\nonumber
		\begin{aligned}
			\norm{f_e}_{Y_0}\le C\norm{h}_{L^2}.
		\end{aligned}
	\end{equation}
	However, for the existence of  the solution, we need to establish the following $L^2$ boundedness
	\begin{equation}\nonumber
		\begin{aligned}
			\norm{f_e}_{L^2}\lesssim  \norm{h}_{L^2},
		\end{aligned}
	\end{equation}
	which is the key difference and difficulty compared with the work \cite{IM2024}. We observe that the homogeneous solution to the linear operator involves a factor \((i\xi)^{-\frac{1}{3}}\), which exhibits a singularity near \(\xi = 0\). 
	For any inhomogeneous term \(h\), the general solution includes a term of the form 
	$$(i\xi)^{-\frac{1}{3}} \int G(y,z;\xi)\hat{h}(\xi,z)\,dz,$$ 
	where the quantity \((i\xi)^{-\frac{1}{3}}G(y,z;\xi)\) represents a specific Green's function. Note that the nonlinear term \(-u\partial_x\omega - v\partial_y\omega\) cannot be entirely recast solely as $x$-directional derivatives (i.e., $\p_x(\text{Terms})$). This structural constraint hinders low-order estimates derived from the nonlinear terms. To address the singularity near $\xi=0$, {\bf a key observation} is that the function
	$$
	t(z) =  \frac{Ai(z)}{Bi(z)} - \frac{Ai'(0)}{Bi'(0)}
	$$
	has no zeros in the sector $\{z\in\C|\ \text{Arg}(z)\in(-\frac\pi3,\frac\pi3)\}$, where $Ai(z),Bi(z)$ denote two independent Airy functions. This property is non-trivial (see Remark \ref{rem for sm}), which involves the distribution of zeros of the Bessel function. Based on this property, we find a subtle function $G$ to create $(i\xi)^{\frac13}$, which takes the following form
	\begin{equation}\nonumber
		\begin{aligned}
			&G(y,z;\xi) = \frac{\pi}{t((i\xi)^{\frac13}M)} \left\{\begin{aligned}
				\mu_M(\xi,y)\mu_0(\xi,z),&\quad 0\le z<y,\\
				\mu_M(\xi,z)\mu_0(\xi,y),&\quad y< z\le M,
			\end{aligned}\right.\\
			&\mu_M(\xi,y) = Ai((i\xi)^{\frac{1}{3}}y) - \frac{Ai((i\xi)^{\frac{1}{3}}M)}{Bi((i\xi)^{\frac{1}{3}}M)}Bi((i\xi)^{\frac{1}{3}}y),\\
			&\mu_0(\xi,y) = Ai((i\xi)^{\frac{1}{3}}y) - \frac{Ai'(0)}{Bi'(0)}Bi((i\xi)^{\frac{1}{3}}y).
		\end{aligned}
	\end{equation}
	It is easy to verify that $\mu_M$ eliminates the singularity near $\xi=0$ and $G$ is bounded on $\big\{z\in\C|\ \text{Arg}(z)\in[-\frac\pi6,\frac\pi6]\big\}$. \smallskip


	The last key point is to show that the displacement $A$ is bounded for all $M\ge M_0$. We observe that $A = \int_0^{M}\omega(x,y)dy$ and then 
	$$
	\norm{A}_{L^{\infty}(\R)} \le \norm{\omega}_{L_x^{\infty}L_y^1}.
	$$
	Indeed, we can bound the $L_x^{\infty}L_y^1$ norm of $\omega$ by means of the norm $\norm{\omega}_{L^2\cap Y_{\bar{\alpha}}}$ for $\bar{\alpha}\ge0$ via interpolation: there exists some \(\theta_1 = \theta_1(\bar{\alpha}) \in (0,1)\),    
	$$
	\norm{\omega}_{L_x^{\infty}L_y^1}\le C(\bar{\alpha},M)\norm{\omega}_{L^2}^{\theta_1}\norm{\omega}_{Y_0}^{1-\theta_1} + \norm{\omega}_{L_x^{\infty}(\R; L_y^1([0,1])}\le C_0(\bar{\alpha},M)\norm{F}_{H^2(\R)}.
	$$
	Here $\bar{\alpha}\in[0,\frac73]$ and $M\ge M_0$.  We observe that
	\begin{equation}
		\inf_{\bar{\alpha}\in[0,\frac73]}C_0(\bar{\alpha},M) = C_0(0,M) = O([\log(M)]^{\frac12}),\quad \forall M\ge M_0.
	\end{equation}
	This means that the bound of $A$ cannot be controlled uniformly in $M$ by the norm $\norm{F}_{H^2(\R)}$—a contradiction to the physical intuition that that the linear component $(\omega_0,A_0)$ of the solution $(\omega,A)$ dictates the behavior of $(\omega,A)$. To eliminate the dependence of the constant $C_0$ on $M$, we introduce a {\bf new norm $\norm{|\p_x|^{\frac{1}{18}}y^{\frac16}\omega}_{L^2}$}, which allows us to derive a refined estimate. Specifically, we have
	$$
	\norm{\omega}_{L_x^{\infty}L_y^1}\le C\big(\norm{|\p_x|^{\frac{1}{18}}y^{\frac16}\omega}_{L^2} + \norm{\omega}_{Y_0}\big)\le C\norm{F}_{H^2(\R)}.
	$$
	Here $C>0$ is independent of $M$. The estimate for this new norm is established by leveraging additional properties of Airy functions (see Proposition \ref{prop for forced terms}).

	\subsection{Notations and Organization of the paper}
	
	We adopt standard mixed-norms on \(\Omega_M\) as follows
	\begin{equation}\nonumber
		\begin{aligned}
			&\langle f,g\rangle_{L^2(\Omega_M)} = \int_{\Omega_M}f\bar{g}\ d\fx,\quad \norm{f}_{L^2}^2 = \int_{\Omega_M}|f|^2d\fx,\\
			&\norm{f}_{L_x^pL_y^q} = \norm{f}_{L_x^p(\R; L_y^q([0,M]))},\quad \mbox{for }\ p,q\ge1,\\
			&\norm{f}_{H_x^{s}L_y^2} = \norm{f}_{H_x^s(\R; L_y^2([0,M]))},\quad
			\norm{f}_{L_x^{2}H_y^s} = \norm{f}_{L_x^2(\R; H_y^s([0,M]))},\quad \mbox{for }\ s\in\R.
		\end{aligned}
	\end{equation}
	
	For some inequalities, the notation  $B_0\lesssim B_1$ denotes that there exists some $C>0$ independent of $B_0$ and $B_1$ such that $B_0\le CB_1$.

	\medskip

	This paper is organized as follows. First, in Section 2, we characterize the functional spaces used in the analysis of the triple-deck equations and derive key nonlinear estimates, including weighted Sobolev norms for vorticity and displacement. Subsequently, Section 3 focuses on the solvability of the linear system of the triple-deck system.  Thereafter, in Section 4, we prove the main result (Theorem \ref{Main thm0}) via a Picard iteration scheme. Finally, we present an appendix, which provides a formal derivation of the complete triple-deck system, including asymptotic matching conditions between the viscous sublayer, main layer, and upper layer. 
	
	\section{Functional spaces and nonlinear estimates}
	
	To study the quantity $\omega$, we first consider the space $X_0\cap X_{\alpha}$. Note that the space $X_{\bar{\alpha}}$ consists of a low-order part and a high-order part. The low-order part is characterized by the norms $\norm{f}_{L^2(\Omega_M)}$ and $\norm{|\p_x|^{\frac{1}{18}}y^{\frac{1}{6}}f}_{L^2(\Omega_M)}$, while the high-order part is characterized by $\norm{f}_{Y_{\bar{\alpha}}}$. We now present some properties between these two parts for $f\in X_0$.
	
	\begin{lem}\label{prop for estimate of f in X0}
		It holds that  for any $0\le\e\le\frac{1}{6}$, 
		\begin{equation}\nonumber
			\begin{aligned}
				\norm{f}_{L_x^2L_y^{\infty}}+ \norm{|\p_x|^{\frac{1}{2}-\e}y^{\frac{1}{2}-3\e}\p_yf}_{L^2(\Omega_M)} + \norm{\p_yf}_{L_x^6L_y^{2}} + \norm{f}_{L_y^1L_x^{\infty}}\le C\norm{f}_{X_0}.
			\end{aligned}
		\end{equation}
		Here $C$ is independent of $\e$.
	\end{lem}
	\begin{proof}
		By Sobelev embedding theorem and the interpolation, we obtain 
		\begin{equation}\nonumber
			\begin{aligned}
				&\norm{f}_{L_x^2L_y^{\infty}}\le C\norm{f}_{L_x^2H_y^2}\le C\norm{f}_{X_0},\\
				&\norm{|\p_x|^{\frac{1}{2}-\e}y^{\frac{1}{2}-3\e}\p_y f}_{L^2(\Omega_M)}\le \norm{|\p_x|^{\frac{1}{3}}\p_y f}_{L^2(\Omega_M)}^{6\e}\norm{|\p_x|^{\frac{1}{2}}y^{\frac{1}{2}}\p_y f}_{L^2(\Omega_M)}^{1-6\e}\le \norm{f}_{Y_0}.
			\end{aligned}
		\end{equation}
		Using the inequality $\norm{|\p_x|^{-\frac{1}{3}}f}_{L_x^{6}}\le C\norm{f}_{L_x^2}$, we derive
		\begin{equation}\nonumber
			\begin{aligned}
				\norm{\p_yf}_{L_x^6L_y^2} \le  C\norm{|\p_x|^{\frac{1}{3}}\p_yf}_{L^2(\Omega_M)}.
			\end{aligned}
		\end{equation}
		Having established the estimates above,  we now turn to the $L_y^1L_x^{\infty}$-estimate of $f$. Through direct calculations, we obtain
		\begin{equation}\nonumber
			\begin{aligned}
				&\norm{f}_{L_x^{\infty}(\R)}\le C\min\bigg\{\norm{y^{\frac{1}{6}}|\p_x|^{\frac{1}{18}}f}_{L_x^2(\R)}^{\frac{9}{17}}\norm{y\p_xf}_{L_x^2}^{\frac{8}{17}}\cdot y^{-\frac{19}{34}},\\
				&\qquad\qquad\qquad\qquad\quad\ \norm{y^{\frac{1}{6}}|\p_x|^{\frac{1}{18}}f}_{L_x^2(\R)}^{\frac{3}{11}}\norm{|\p_x|^{\frac{2}{3}}f}_{L_x^2}^{\frac{8}{11}}\cdot y^{-\frac{1}{22}}\bigg\},\\
				&\norm{f}_{L_y^1L_x^{\infty}}\le C\norm{y^{\frac{1}{6}}|\p_x|^{\frac{1}{18}}f}_{L^2(\Omega_M)}^{\frac{3}{11}}\norm{|\p_x|^{\frac{2}{3}}f}_{L^2(\Omega_M)}^{\frac{8}{11}}\norm{y^{-\frac{1}{22}}}_{L_y^2([0,1])}\\
				&\qquad\qquad\qquad + C\norm{y^{\frac{1}{6}}|\p_x|^{\frac{1}{18}}f}_{L^2(\Omega_M)}^{\frac{9}{17}}\norm{y\p_xf}_{L^2(\Omega_M)}^{\frac{8}{17}}\norm{y^{-\frac{19}{34}}}_{L_y^2([1,\infty])}\\
				&\qquad\qquad\le C\norm{y^{\frac{1}{6}}|\p_x|^{\frac{1}{18}}f}_{L^2(\Omega_M)}^{\frac{3}{11}}\norm{|\p_x|^{\frac{2}{3}}f}_{L^2(\Omega_M)}^{\frac{8}{11}}\\
				&\qquad\qquad\qquad + C\norm{y^{\frac{1}{6}}|\p_x|^{\frac{1}{18}}f}_{L^2(\Omega_M)}^{\frac{9}{17}}\norm{y\p_xf}_{L^2(\Omega_M)}^{\frac{8}{17}}\\
				&\qquad\qquad\le C\norm{f}_{X_0}.
			\end{aligned}
		\end{equation}
	\end{proof}
	
	
	\begin{lem}\label{prop for estimate of f in Xalpha}
		Assume that $f\in X_0\cap X_\alpha$ for $\alpha>0$. Then for any $\e\in(0,\min\{\frac{1}{6},\frac{\alpha}{2}\}]$, we have
		\begin{equation}\nonumber 
			\begin{aligned}
				\norm{y^{\frac{1-3\e}{2}}\p_yf}_{L_y^2L_x^{\infty}} + \norm{f}_{L^{\infty}(\Omega_M)}\le C\Big(\frac{(1+\alpha)^{\frac{1}{2}}}{\alpha}+\frac{1}{\sqrt{\e}}\Big)\big(\norm{f}_{X_0} + \norm{f}_{X_{\alpha}}\big).
			\end{aligned}
		\end{equation}
		Here $C>0$ is independent of $\e,\alpha$.
	\end{lem}
	
	\begin{proof}
		By the interpolation and Lemma \ref{prop for estimate of f in X0}, we have
		\begin{equation}\nonumber
			\begin{aligned}
				&\norm{\p_yf}_{L_x^{\infty}(\R)}\le \frac{C}{\sqrt{\e}}\norm{|\p_x|^{\frac{1}{2}-\e}\p_yf}_{L_x^2(\R)}^{\frac{1}{2}}\norm{|\p_x|^{\frac{1}{2}+ \e}\p_yf}_{L_x^2(\R)}^{\frac{1}{2}},\\ 
				&\norm{y^{\frac{1-3\e}{2}}\p_yf}_{L_y^2L_x^{\infty}}\le \frac{C}{\sqrt{\e}}\norm{|\p_x|^{\frac{1}{2}+\e}y^{\frac{1}{2}}\p_yf}_{L_y^2L_x^2}^{\frac{1}{2}} \norm{|\p_x|^{\frac{1}{2}-\e}y^{\frac{1}{2}-3\e}\p_yf}_{L_y^2L_x^2}^{\frac{1}{2}}\\
				&\qquad\qquad\qquad\quad\ \le \frac{C}{\sqrt{\e}}\Big(\norm{|\p_x|^{\frac{1}{2}}y^{\frac{1}{2}}\p_yf}_{L^2(\Omega_M)}^{\frac{1}{2}} + \norm{|\p_x|^{\frac{1+\alpha}{2}}y^{\frac{1}{2}}\p_yf}_{L^2(\Omega_M)}^{\frac{1}{2}}\Big) \norm{f}_{Y_0}^{\frac{1}{2}}\\
				&\qquad\qquad\qquad\quad\ \le \frac{C}{\sqrt{\e}}\big(\norm{f}_{Y_0} + \norm{f}_{Y_{\alpha}}\big).
			\end{aligned}
		\end{equation}
		We denote by $\tilde{f}$ the extension of $f$ to $\R^2$ such that $\tilde{f}\in H_y^2(\R;L_x^2(\R))\cap H_x^{\frac{2}{3}+\frac{\alpha}{2}}(\R;L_y^2(\R))$ for $\alpha>0$ and
		\begin{equation}\nonumber
			\begin{aligned}
				&\norm{\tilde{f}}_{H_y^2(\R;L_x^2(\R))}\le C\norm{f}_{H_y^2((0,M);L_x^2(\R))},\\
				&\norm{\tilde{f}}_{H_x^{\frac{2}{3}+\frac{\alpha}{2}}(\R;L_y^2(\R))}\le C\norm{f}_{H_x^{\frac{2}{3}+\frac{\alpha}{2}}(\R;L_y^2(0,M))}.
			\end{aligned}
		\end{equation}
		Direct calculations show that
		\begin{equation}\nonumber
			\begin{aligned}
				\norm{\tilde{f}}_{L_{x,y}^{\infty}}\le&\ \norm{\hat{\tilde{f}}}_{L_{\xi,\eta}^1}\\
				=&\ \int_{\R^2}\bigg\{\big|\hat{\tilde{f}}(\xi,\eta)(1+|\eta|)^2\big|^{\theta}\big|\hat{\tilde{f}}(\xi,\eta)(1+|\xi|)^{\frac{2}{3}+\frac{\alpha}{2}}\big|^{1-\theta}\\
				&\qquad\quad\cdot (1+|\eta|)^{-2\theta}(1+|\xi|)^{-(\frac{2}{3}+\frac{\alpha}{2})(1-\theta)}\bigg\}d\xi d\eta\\
				\le&\ C(\alpha,\theta)\norm{\big|\hat{\tilde{f}}(\xi,\eta)(1+|\eta|)^2\big|^{\theta}}_{L_{\xi,\eta}^{\frac{2}{\theta}}} \norm{\big|\hat{\tilde{f}}(\xi,\eta)(1+|\xi|)^{\frac{2}{3}+\frac{\alpha}{2}}\big|^{1-\theta}}_{L_{\xi,\eta}^{\frac{2}{1-\theta}}}\\
				\le&\ C(\alpha,\theta)\norm{\tilde{f}}_{H_y^2(\R;L_x^2(\R))}^{\theta} \norm{\tilde{f}}_{H_x^{\frac{2}{3}+\frac{\alpha}{2}}(\R;L_y^2(\R))}^{1-\theta}\\
				\le&\ C\frac{(1+\alpha)^{\frac{1}{2}}}{\alpha}\big(\norm{\tilde{f}}_{H_y^2(\R;L_x^2(\R))} + \norm{\tilde{f}}_{H_x^{\frac{2}{3}+\frac{\alpha}{2}}(\R;L_y^2(\R))}\big).
			\end{aligned}
		\end{equation}
		Here $C(\alpha,\theta) = \norm{(1+|\eta|)^{-2\theta}(1+|\xi|)^{-(\frac{2}{3} + \frac{\alpha}{2})(1-\theta)}}_{L_{\xi,\eta}^2}$ and $\theta=\frac{1}{2}(\frac{1}{4}+\frac{1+3\alpha}{4+3\alpha})$.
		This shows that
		\begin{equation}\nonumber
			\begin{aligned}
				\norm{f}_{L_{x,y}^{\infty}}\le&\ \norm{\tilde{f}}_{L_{x,y}^{\infty}}\le C\frac{(1+\alpha)^{\frac{1}{2}}}{\alpha}\big(\norm{\tilde{f}}_{H_y^2(\R;L_x^2(\R))} + \norm{\tilde{f}}_{H_x^{\frac{2}{3}+\frac{\alpha}{2}}(\R;L_y^2(\R))}\big)\\
				\le&\ C\frac{(1+\alpha)^{\frac{1}{2}}}{\alpha}\big(\norm{f}_{H_y^2((0,M);L_x^2(\R))} + \norm{f}_{H_x^{\frac{2}{3}+\frac{\alpha}{2}}\big(\R;L_y^2((0,M))\big)}\big)\\
				\le&\ C\frac{(1+\alpha)^{\frac{1}{2}}}{\alpha}\big(\norm{f}_{X_0} + \norm{f}_{X_{\alpha}}\big).
			\end{aligned}
		\end{equation}
	\end{proof}

	Next we are devoted to the estimates of nonlinear terms. Let the velocity $(u,v)$ satisfy the divergence-free condition
	\begin{align*}
		\p_xu + \p_yv = 0
	\end{align*}
	subject to the boundary conditions
	\begin{align*}
		(u,v)\big|_{y=0} = 0,\quad u\big|_{y=M} =A_M.
	\end{align*}
	where $A_M$ denotes a fixed velocity at the upper boundary $y=M$. Define the function 
	$$
	h = -(u\p_x\omega + v\p_y\omega),
	$$
	for $\omega = \p_yu$. Using the divergence-free condition and boundary conditions, we can rewrite $h$ as
	\begin{equation}\nonumber
		\begin{aligned}
			h = -u\p_x\omega - (V^0 - y\p_xA_M)\p_y\omega.
		\end{aligned}
	\end{equation}
	Here 
	\begin{align}\label{de:V0}
		V^0=\int_0^y\int_z^{M}\p_x\omega(\cdot,\tilde{z})d\tilde{z}dz.
	\end{align}

	\begin{prop}\label{prop for forced terms}
		Let $\alpha>\frac{2}{3},\ \e\in(0,\frac{1}{6}]$. Assume that $\omega\in X_0\cap X_{\alpha}$ and $ \p_xA_M\in L^{\infty}\cap H^{\frac{\alpha}{2}}$. Then $h\in H_x^{\frac{\alpha}{2}}L_y^2$ with
		\begin{equation}\label{2.12}
			\begin{aligned}
				&\norm{h}_{L^2(\Omega_M)}\le CM^{\frac{1+3\e}{2}}\big(\norm{\omega}_{X_0\cap X_{\alpha}} + \norm{\p_xA_M}_{L^2(\R)}\big)\norm{\omega}_{X_0\cap X_{\alpha}},\\
				&\norm{|\p_x|^{\frac{\alpha}{2}}h}_{L^2(\Omega_M)}\le CM\big(\norm{\omega}_{X_0\cap X_{\alpha}} + \norm{\p_xA_M}_{L^{\infty}(\R)\cap \dot{H}^{\frac{\alpha}{2}}(\R)}\big)\norm{\omega}_{X_0\cap X_{\alpha}}.
			\end{aligned}
		\end{equation}
		Here $C=C(\alpha,\e)>0$ is independent of $M$ and $A_M$.
	\end{prop}
	
	\begin{proof}
		By virtue of Lemma \ref{prop for estimate of f in X0} and Lemma \ref{prop for estimate of f in Xalpha}, we obtain
		\begin{equation*}
			\begin{aligned}
				\norm{h}_{L^2(\Omega_M)}\le&\ C\norm{\frac{u}{y}}_{L^{\infty}(\Omega_M)}\norm{y\p_x\omega}_{L^2(\Omega_M)} + CM^{\frac{3\e}{2}}\norm{\tilde{z}\p_x\omega}_{L^2(\Omega_M)}\norm{y^{\frac{1-3\e}{2}}\p_y\omega}_{L_y^2L_x^{\infty}}\\
				&\ \ + CM^{\frac{1+3\e}{2}}\norm{\p_xA_M}_{L^2(\R)}\norm{y^{\frac{1-3\e}{2}}\p_y\omega}_{L_y^2L_x^{\infty}}\\
				\le&\ C\big(\norm{\omega}_{L^{\infty}(\Omega_M)} + M^{\frac{3\e}{2}}\norm{y^{\frac{1-3\e}{2}}\p_y\omega}_{L_y^2L_x^{\infty}}\big)\norm{y\p_x\omega}_{L^2(\Omega_M)}\\
				&\ \ + CM^{\frac{1+3\e}{2}}\norm{\p_xA_M}_{L^{2}(\R)}\norm{y^{\frac{1-3\e}{2}}\p_y\omega}_{L_y^2L_x^{\infty}}\\
				\le&\ C(1+M^{\frac{1+3\e}{2}})\norm{\omega}_{X_0\cap X_{\alpha}}^2 + CM^{\frac{1+3\e}{2}}\norm{\p_xA_M}_{L^2(\R)}\norm{\omega}_{X_0\cap X_{\alpha}},
			\end{aligned}
		\end{equation*}
		which gives the first estimate of  \eqref{2.12}.

		For the higher-order estimates, by product estimates, we have
		\begin{equation}\label{higher estimates0}
			\begin{aligned}
				&\norm{|\p_x|^{\frac{\alpha}{2}}(u\p_x\omega)}_{L^2(\Omega_M)}\\
				&\le C\norm{\omega}_{L^{\infty}(\Omega_M)}\norm{|\p_x|^{\frac{\alpha}{2}}y\p_x\omega}_{L^2(\Omega_M)} + C\norm{|\p_x|^{\frac{\alpha}{2}}\omega}_{H_x^{\frac{2}{3}}L_y^2}\norm{y|\p_x|^{\frac{1}{3}}\omega}_{L^{\infty}(\Omega_M)}\\
				&\le C\norm{\omega}_{X_0\cap X_{\alpha}}^2 + C\norm{|\p_x|^{\frac{\alpha}{2}}\omega}_{H_x^{\frac{2}{3}}L_y^2}\norm{y|\p_x|^{\frac{1}{3}}\omega}_{L^{\infty}(\Omega_M)},
			\end{aligned}
		\end{equation}
		and  
		\begin{equation}\label{higher estimates1}
			\begin{aligned} 
				&\norm{|\p_x|^{\frac{\alpha}{2}}(V^0 \p_y\omega)}_{L^2(\Omega_M)}\\
				&\le C\norm{y^{-\frac{1}{2}}\int_0^y\int_z^{M}|\p_x|^{\frac{\alpha}{2}}\p_x\omega(\cdot,\tilde{z})d\tilde{z}dz}_{L_x^2L_y^{\infty}}\norm{y^{\frac{1}{2}}\p_y\omega}_{L_x^{\infty}L_y^2}\\
				&\qquad+ C\norm{y^{-\frac{1}{2}}\int_0^y\int_z^{M}\p_x\omega(\cdot,\tilde{z})d\tilde{z}dz}_{L_x^4L_y^{\infty}}\norm{|\p_x|^{\frac{\alpha}{2}}y^{\frac{1}{2}}\p_y\omega}_{L_x^{4}L_y^2}\\
				&\le CM^{\frac{3\e}{2}}\norm{|\p_x|^{\frac{\alpha}{2}}\tilde{z}\p_x\omega}_{L^2(\Omega_M)}\norm{y^{\frac{1-3\e}{2}}\p_y\omega}_{L_y^2L_x^{\infty}}\\
				&\qquad+C\norm{\tilde{z}\p_x\omega}_{L_x^4L_{\tilde{z}}^2}\norm{|\p_x|^{\frac{\alpha}{2}}y^{\frac{1}{2}}\p_y\omega}_{L_x^{4}L_y^2}\\
				&\le CM^{\frac{3\e}{2}}\norm{\omega}_{X_0\cap X_{\alpha}}^2 + C\norm{\tilde{z}\p_x\omega}_{L_x^4L_{\tilde{z}}^2}\norm{|\p_x|^{\frac{\alpha}{2}}y^{\frac{1}{2}}\p_y\omega}_{L_x^{4}L_y^2} .
			\end{aligned}
		\end{equation}
		
		It remains to estimate the right-hand side of the above inequality.  By the Sobolev embedding, we have
		\begin{equation}\label{higher omega estimates0}
			\begin{aligned}
				&\norm{|\p_x|^{\frac{\alpha}{2}}\omega}_{H_x^{\frac{2}{3}}L_y^2}\le C\norm{|\p_x|^{\frac{\alpha}{2}}\omega}_{L_x^2L_y^2} + C\norm{|\p_x|^{\frac{2}{3}+\frac{\alpha}{2}}\omega}_{L_x^2L_y^2}\\
				&\qquad\le C\norm{\omega}_{L_x^2L_y^2} + \norm{|\p_x|^{\frac{2}{3}+\frac{\alpha}{3}}\omega}_{L_x^2L_y^2}\le C\norm{\omega}_{X_{\alpha}},\\
				&\norm{y|\p_x|^{\frac{1}{3}}\omega}_{L^{\infty}(\Omega_M)}\le\ M\norm{|\p_x|^{\frac{1}{3}}\omega}_{L^{\infty}(\Omega)}\\
				&\qquad\le C\frac{[1+(\alpha-\frac{2}{3})]^{\frac{1}{2}}}{\alpha-\frac{2}{3}}M\norm{|\p_x|^{\frac{1}{3}}\omega}_{H_x^{\frac{2}{3}+(\frac{\alpha}{2}-\frac{1}{3})}L_y^2\cap L_x^2H_y^2}\quad \mbox{for}\ \ \alpha>\frac{2}{3}\\
				&\qquad\le CM\big(\norm{|\p_x|^{\frac{1}{3}}\omega}_{L_x^2L_y^2} + \norm{|\p_x|^{\frac{2}{3}+\frac{\alpha}{2}}\omega}_{L_x^2L_y^2} + \norm{|\p_x|^{\frac{1}{3}}\p_{y}^2\omega}_{L_x^2L_y^2}\big)\\
				&\qquad\le CM\big(\norm{\omega}_{L_x^2L_y^2} + \norm{|\p_x|^{\frac{2}{3}+\frac{\alpha}{2}}\omega}_{L_x^2L_y^2} + \norm{\p_{y}^2\omega}_{L_x^2L_y^2} + \norm{|\p_x|^{\frac{\alpha}{2}}\p_{y}^2\omega}_{L_x^2L_y^2}\big)\\
				&\qquad\le CM\norm{\omega}_{X_0\cap X_{\alpha}},\\
			\end{aligned}
		\end{equation}
		and
		\begin{equation}
			\begin{aligned}
				\norm{y\p_x\omega}_{L_x^4L_y^2}&\le C\norm{|\p_x|^{\frac{1}{4}}y\p_x\omega}_{L_x^2L_y^2}\\
				&\le C\big(\norm{y\p_x\omega}_{L_x^2L_y^2} + \norm{|\p_x|^{\frac{\alpha}{2}}y\p_x\omega}_{L_x^2L_y^2}\big)\quad \mbox{for}\ \ \alpha\ge\frac{1}{2}\\
				&\le C\norm{\omega}_{X_0\cap X_{\alpha}},\\
				\norm{|\p_x|^{\frac{\alpha}{2}}y^{\frac{1}{2}}\p_y\omega}_{L_x^4L_y^2}&\le \norm{|\p_x|^{\frac{\alpha}{2}+\frac{1}{4}}y^{\frac{1}{2}}\p_y\omega}_{L_x^2L_y^2}\\
				&\le C\big(\norm{|\p_x|^{\frac{1}{2}}y^{\frac{1}{2}}\p_y\omega}_{L_x^2L_y^2} + \norm{|\p_x|^{\frac{\alpha}{2}+\frac{1}{2}}y^{\frac{1}{2}}\p_y\omega}_{L_x^2L_y^2}\big)\\
				&\le C\norm{\omega}_{X_0\cap X_{\alpha}}.
			\end{aligned}
		\end{equation}
		Combining all estimates above, we deduce that for $\alpha>\frac{2}{3}$, 
		\begin{equation}
			\begin{aligned}
				&\norm{|\p_x|^{\frac{\alpha}{2}}(u\p_x\omega)}_{L^2(\Omega_M)}\le\ CM\norm{\omega}_{X_0\cap X_{\alpha}}^2,\\
				&\norm{|\p_x|^{\frac{\alpha}{2}}(V^0\p_y\omega)}_{L^2(\Omega_M)}\le\ CM^{\frac{3\e}{2}}\norm{\omega}_{X_0\cap X_{\alpha}}^2.
			\end{aligned}
		\end{equation}
		Moreover, we have
		\begin{equation}\label{higher pxA estimates0}
			\begin{aligned}
				&\norm{|\p_x|^{\frac{\alpha}{2}}(y\p_xA_M\p_y\omega)}_{L^2(\Omega_M)}\\
				& \le C\norm{|\p_x|^{\frac{\alpha}{2}}\p_xA_M}_{L^2(\R)}\norm{y\p_y\omega}_{L_x^{\infty}L_y^2} + C\norm{\p_xA_M}_{L^{\infty}(\R)}\norm{|\p_x|^{\frac{\alpha}{2}}y\p_y\omega}_{L_x^2L_y^2}\\
				&\le CM^{\frac{1+3\e}{2}}\norm{\omega}_{X_0\cap X_{\alpha}}\norm{\p_xA_M}_{\dot{H}^{\frac{\alpha}{2}}} + CM\norm{\omega}_{X_0\cap X_{\alpha}}\norm{\p_xA_M}_{L^{\infty}}\\
				&\le CM\norm{\omega}_{X_0\cap X_{\alpha}}\norm{\p_xA_M}_{L^{\infty}\cap\dot{H}^{\frac{\alpha}{2}}}.
			\end{aligned}
		\end{equation}
		
		Substituting \eqref{higher omega estimates0}-\eqref{higher pxA estimates0} into \eqref{higher estimates0}-\eqref{higher estimates1}, we conclude the last estimate of \eqref{2.12}.
	\end{proof}

	\begin{rem}\label{rem for forced terms}
		Indeed, if $h$ is replaced by $h_1=-(u\p_x\omega_1 + v\p_y\omega_1)$ for $\omega_1\in X_0\cap X_{\alpha}$, then by repeating the proof of Proposition \ref{prop for forced terms}, we have
		\begin{equation}\nonumber
			\begin{aligned}
				&\norm{h_1}_{L^2(\Omega_M)}\le CM^{\frac{1+3\e}{2}}\big(\norm{\omega}_{X_0\cap X_{\alpha}} + \norm{\p_xA_M}_{L^2(\R)}\big)\norm{\omega_1}_{X_0\cap X_{\alpha}},\ \ \e\in(0,\frac{1}{6}],\\
				&\norm{|\p_x|^{\frac{\alpha}{2}}h_1}_{L^2(\Omega_M)}\le CM\big(\norm{\omega}_{X_0\cap X_{\alpha}} + \norm{\p_xA_M}_{L^{\infty}(\R)\cap \dot{H}^{\frac{\alpha}{2}}(\R)}\big)\norm{\omega_1}_{X_0\cap X_{\alpha}}.
			\end{aligned}
		\end{equation}
	\end{rem}

	The following lemma is devoted to the integral boundary condition $I_M[\omega] = A(x) + F(x)$.
	\begin{lem}\label{prop for IMf}
		It holds that for any $\bar{\alpha}\ge0$, 
		\begin{equation}
			\begin{aligned}
				\norm{|\p_x|^{\frac{5}{6} + \frac{\bar{\alpha}}{2}}I_{M}[f]}_{L^2(\R)}\le&\ C\norm{f}_{Y_{\bar{\alpha}}}.
			\end{aligned}
		\end{equation}
		Here $C$ is independent of $\bar{\alpha}$.
	\end{lem}
	
	\begin{proof}
		Observe that
		\begin{equation*}
			\begin{aligned}
				\left|\int_0^{M}\hat{f}(\xi,y)dy\right| =&\left|\int_0^{|\xi|^{-\frac{1}{3}}}\hat{f}(\xi,y)dy + \int_{|\xi|^{-\frac{1}{3}}}^{M}y^{-1}y\hat{f}(\xi,y)dy\right|\\
				\le&\ C|\xi|^{-\frac{1}{6}}\norm{\hat{f}(\xi,\cdot)}_{L_y^2} + C|\xi|^{\frac{1}{6}}\norm{y\hat{f}(\xi,\cdot)}_{L_y^2},
			\end{aligned}
		\end{equation*}
		which implies 
		\begin{equation}\nonumber
			\begin{aligned}
				\big|\xi^{\frac{5}{6}+\frac{\bar{\alpha}}{2}} I_{M}[f]\big|\le&\ C|\xi|^{\frac{2}{3} + \frac{\bar{\alpha}}{2}}\cdot\norm{\hat{f}(\xi,\cdot)}_{L_y^2(\R)} + C|\xi|^{1+\frac{\bar{\alpha}}{2}}\norm{y\hat{f}(\xi,\cdot)}_{L_y^2}).
			\end{aligned}
		\end{equation}
		Thus, we obtain
		\begin{equation}\nonumber
			\begin{aligned}
				\norm{|\p_x|^{\frac{5}{6} + \frac{\bar{\alpha}}{2}}I_{M}[f]}_{L^2(\R)}\le C\big(\norm{|\p_x|^{\frac{2}{3} +\frac{\bar{\alpha}}{2}}f}_{L^2} + \norm{y|\p_x|^{\frac{\bar{\alpha}}{2}}\p_xf}_{L^2}\big)\le C\norm{f}_{Y_{\bar{\alpha}}}.
			\end{aligned}
		\end{equation}
		It completes the proof.
	\end{proof}

	\section{The solvability of the linear system}
	
	Recall that, to solve the nonlinear system \eqref{Perturbation system0}, we decompose $\omega = \omega_0 + \bar{\omega}$ and $A = A_0 + \bar{A}$ with 
	\begin{equation}\label{Eq for omega0-1}
		\left\{\begin{aligned}
			&y\p_x\omega_0 - \p_y^2\omega_0 = 0,\quad (x,y)\in\Omega_M,\\
			&\p_y\omega_0\big|_{y=0} = \p_x|\p_x|A_0,\quad I_M[\omega_0] = A_0 + F,\quad x\in\R,
		\end{aligned}\right.
	\end{equation}
	and
	\begin{equation}\label{Eq for omega1-1}
		\left\{\begin{aligned}
			&y\p_x\bar{\omega} - \p_y^2\bar{\omega} = -(u\p_x\omega + v\p_y\omega)\triangleq \bh,\quad (x,y)\in\Omega_M,\\
			&u = I_y[\omega],\quad v = -\p_xI_y[u],\quad (x,y)\in\Omega_M,\\
			&\p_y\bar{\omega}\big|_{y=0} = \p_x|\p_x|\bar{A},\quad I_M[\bar{\omega}] = \bar{A},\quad x\in\R.
		\end{aligned}\right.
	\end{equation}
	
	To address the system \eqref{Eq for omega1-1}, we first solve an inhomogeneous linear problem without the integral condition
	\begin{equation}\label{Linearized PDE0}
		\left\{\begin{aligned}
			&y\p_xf_e - \p_y^2f_e = h_e,\quad (x,y)\in \Omega_M,\\
			&\p_yf_e\big|_{y=0} = 0,\quad f_e\big|_{y=M}=0,\quad x\in\R.
		\end{aligned}\right.
	\end{equation}
	
	\begin{prop}\label{lem for nonlinear terms}
		Assume that $h_e\in H_x^{\frac{\alpha}{2}}L_y^2$ for some $\alpha>0$. Then there exists a unique solution $f_e\in X_0\cap X_{\alpha}$ to the problem \eqref{Linearized PDE0}.
		Moreover, there holds that
		\begin{equation*}
			\norm{f_e}_{X_0\cap X_{\alpha}}\le C\big[(1+M^2)\norm{h_e}_{L^2(\Omega_M)} + \norm{|\p_x|^{\frac{\alpha}{2}}h_e}_{L^2(\Omega_M)}\big],
		\end{equation*}
		where $C$ is a constant independent of $\alpha, M$.
	\end{prop}
	
	\begin{proof}
		The proof proceeds in three steps. \smallskip
		
		\par \textbf{Step 1.} The explicit formula of the solution.\smallskip
		
		Taking the Fourier transform in variable $x$ into \eqref{Linearized PDE0}, we obtain
		\begin{equation}\label{FT of fe}
			\left\{\begin{aligned}
				&i\xi y\hat{f}_e - \p_{y}^2\hat{f}_e = \hat{h}_e,\quad  (\xi,y)\in\R\times(0,M),\\
				&\p_{y}\hat{f}_e\big|_{y=0} = 0,\  \hat{f}_e\big|_{y=M} = 0,\quad \xi\in\R.
			\end{aligned}\right.
		\end{equation}
		Using the Green-function method, we find the unique solution $\hat{f}_e$ of \eqref{FT of fe}, which has an explicit formula 
		\begin{equation}\label{Explicit formula of hatfe}
			\begin{aligned}
				\hat{f}_e(\xi,y) =&\ (i\xi)^{-\frac13}\int_0^M G(y,z; \xi)\hat{h}_e(\xi,z)dz,
			\end{aligned}
		\end{equation}
		with the definition of $G$ given by
		\begin{equation}\label{Green func}
			\begin{aligned}
				&G(y,z;\xi) = \pi s_M \left\{\begin{aligned}
					\mu_M(\xi,y)\mu_0(\xi,z),&\quad 0\le z<y,\\
					\mu_M(\xi,z)\mu_0(\xi,y),&\quad y< z\le M,
				\end{aligned}\right.\\
				&s_M =  \bigg(\frac{Ai((i\xi)^{\frac{1}{3}}M)}{Bi((i\xi)^{\frac{1}{3}}M)} - \frac{Ai'(0)}{Bi'(0)}\bigg)^{-1},\\
				&\mu_M(\xi,y) = Ai((i\xi)^{\frac{1}{3}}y) - \frac{Ai((i\xi)^{\frac{1}{3}}M)}{Bi((i\xi)^{\frac{1}{3}}M)}Bi((i\xi)^{\frac{1}{3}}y),\\
				&\mu_0(\xi,y) = Ai((i\xi)^{\frac{1}{3}}y) - \frac{Ai'(0)}{Bi'(0)}Bi((i\xi)^{\frac{1}{3}}y).
			\end{aligned}
		\end{equation}
		
		\textbf{Step 2.} The low-order estimates of $f_e$. \smallskip

		To this end, we first claim that $s_M$ is bounded. Since the ratio $\frac{Ai'(z)}{Bi'(z)}$ of the Airy functions at $z=0$ is $-\frac{1}{\sqrt{3}}$, we then define
		\begin{equation}\nonumber
			t(z): = \frac{Ai(z)}{Bi(z)} - \frac{Ai'(0)}{Bi'(0)} = \frac{Ai(z)}{Bi(z)} + \frac{1}{\sqrt{3}} = \frac{\sqrt{3}Ai(z) + Bi(z)}{\sqrt{3}Bi(z)}.
		\end{equation}
		It is well-known that $Ai(z)$ and $Bi(z)$ are holomorphic on $\C$ and do not vanish for $\text{Arg}(z)\in(-\frac \pi3,\frac \pi3)$ \cite[Page 15]{VS2010}. For clarity of the proof, we claim that $\sqrt{3}Ai(z) + Bi(z)$ on $\C$ has all its zeros in the rays $\big\{z\in\C|\ \text{Arg}(z)=\frac{\pi}{3}+\frac{2k\pi}{3},\ \ k\in\Z\big\}$ (see details in Remark \ref{rem for sm}). Define a sector with the angle $2\theta$ by
		$$
		C_{\theta}=\big\{z\in\C\big|\ \text{Arg}(z)\in(-\theta,\theta)\big\}.
		$$
		Then, $t(z)$ is a holomorphic and non-vanishing function on $C_{\frac\pi3}$. Moreover, for any $\delta\in(0,\frac\pi3)$, the Airy functions $Ai,Bi$ satisfy the following asymptotic expansions(see \cite[Chapter 9]{Olver2010}):
		\begin{equation}\label{Asym of Ai}
			\begin{aligned}
				&Ai(z)\sim \frac{1}{3^{\frac{2}{3}}\Gamma(\frac{1}{3})} - \frac{1}{3^{\frac{1}{3}}\Gamma(\frac{2}{3})}z,\quad Bi(z)\sim \frac{1}{3^{\frac{1}{6}}\Gamma(\frac{1}{3})} + \frac{3^{\frac{1}{6}}}{\Gamma(\frac{2}{3})}z,\qquad \mbox{as}\ \ z\to 0,\\
				&Ai(z)\sim \frac{e^{-\frac{2}{3}z^{\frac{3}{2}}}}{2\sqrt{\pi}z^{\frac{1}{4}}},\quad Bi(z)\sim \frac{e^{\frac{2}{3}z^{\frac{3}{2}}}}{\sqrt{\pi}z^{\frac{1}{4}}},\qquad \mbox{as}\ \ |z|\to\infty\ \ \big(z\in C_{\frac\pi3-\delta}\big).
			\end{aligned}
		\end{equation}
		By the asymptotic behavior \eqref{Asym of Ai}, we have 
		\begin{equation}\label{Non-vanishing at inf}
			\begin{aligned}
				\lim_{|z|\to\infty,\atop |\text{Arg}(z)|\in [0,\frac\pi3-\delta)}t(z) = \frac{1}{\sqrt{3}}.
			\end{aligned}
		\end{equation}
		Then, combining the holomorphicity and non-vanishing property of $t(z)$ on $C_{\frac\pi3}$ with \eqref{Non-vanishing at inf}, there exists some $c_0>0$ such that
		$$
		|t(z)|\ge c_0^{-1},\quad \text{for }z\in \overline{C_{\frac\pi6}}.
		$$
		Thanks to $s_M = \bigg(\frac{Ai((i\xi)^{\frac{1}{3}}M)}{Bi((i\xi)^{\frac{1}{3}}M)} - \frac{Ai'(0)}{Bi'(0)}\bigg)^{-1} = \frac{1}{t((i\xi)^{\frac13}M)}$, we then obtain
		\begin{equation}\label{sM}
			\begin{aligned}
				|s_M|\le c_0,\qquad \forall\ \xi\in\R,\ M>0.
			\end{aligned}
		\end{equation}
		
		To derive a low-order estimate of $f_e$, we decompose the frequency space into two regimes. For the low-frequency regime where $|\xi|\le M^{-3}$, we observe that $|(i\xi)^{\frac13}y|\le 1$ for all $y\in[0,M]$. By the continuity of $Ai$ and $Bi$ on $\C$, we have
		\begin{equation}
			\begin{aligned}
				&|\mu_0(\xi,y)| = \big|Ai((i\xi)^{\frac{1}{3}}y) + \frac{1}{\sqrt{3}}Bi((i\xi)^{\frac{1}{3}}y)\big|\le C,\quad \forall\ y\in[0,M].
			\end{aligned}
		\end{equation}
		Next, we compute
		\begin{equation}\label{3.10}
			\begin{aligned}
				&\mu_M(\xi,y) = Ai((i\xi)^{\frac{1}{3}}y) - \frac{Ai((i\xi)^{\frac{1}{3}}M)}{Bi((i\xi)^{\frac{1}{3}}M)}Bi((i\xi)^{\frac{1}{3}}y)\\
				&\qquad\quad\ \  = Ai((i\xi)^{\frac{1}{3}}y) - \frac{Ai\big((i\xi)^{\frac{1}{3}}(M-y) + (i\xi)^{\frac13}y\big)}{Bi\big((i\xi)^{\frac{1}{3}}(M-y) + (i\xi)^{\frac13}y\big)}Bi((i\xi)^{\frac{1}{3}}y).
			\end{aligned}
		\end{equation}
		Since $Ai((i\xi)^{\frac{1}{3}}y) - \frac{Ai((i\xi)^{\frac13}y)}{Bi((i\xi)^{\frac13}y)}Bi((i\xi)^{\frac{1}{3}}y) = 0$ and $Bi(z)\not=0$ for ${\rm Arg}(z)\in(-\frac\pi3,\frac\pi3)$, an application of the Mean Value Theorem to \eqref{3.10} yields
		\begin{equation}\label{3.11}
			\begin{aligned}
				&|\mu_M(\xi,y)| \le C|\xi|^{\frac{1}{3}}(M-y),\quad \forall\ y\in[0,M].
			\end{aligned}
		\end{equation}
		It follows that
		\begin{equation}
			\begin{aligned}
				|\hat{f}_e(\xi,y)| =&\ \bigg|\pi s_M(i\xi)^{-\frac13}\bigg\{\mu_M(\xi,y)\cdot\int_0^y \mu_0(\xi,z)\hat{h}_e(\xi,z)dz\\
				&\qquad\qquad\qquad\quad+ \mu_0(\xi,y)\cdot\int_y^M\mu_M(\xi,z)\hat{h}_e(\xi,z)dz\bigg\}\bigg|\\
				\le&\ C|s_M|\cdot|\xi|^{-\frac{1}{3}}\bigg\{\big|\mu_M(\xi,y)\big|\cdot\int_0^y \big|\mu_0(\xi,z)\hat{h}_e(\xi,z)\big|dz\\
				&\qquad\qquad\qquad\quad+ \big|\mu_0(\xi,y)\big|\cdot\int_y^M\big|\mu_M(\xi,z)\hat{h}_e(\xi,z)\big|dz\bigg\}\\
				\le&\ C(M-y)y^{\frac{1}{2}}\norm{\hat{h}_e(\xi,\cdot)}_{L_z^2([0,y])}\quad \text{by H\"older's inequality and }\eqref{3.11}\\
				&\quad + C|\xi|^{-\frac{1}{3}}\norm{|\xi|^{\frac{1}{3}}(M-z)}_{L_z^2([y,M])}\norm{\hat{h}_e(\xi,\cdot)}_{L_z^2([y,M])}\\
				\le&\ C(M-y)M^{\frac{1}{2}}\norm{\hat{h}_e(\xi,\cdot)}_{L_z^2([0,M])},\quad \mbox{for}\ \ y\in[0,M].
			\end{aligned}
		\end{equation}
		Hence, we obtain a point-wise estimate of $\hat{f}_e$ for the low frequency as follows
		\begin{equation}\label{Pointwise estimate in xi0}
			\begin{aligned}
				\norm{(|\xi|^{\frac{1}{3}}y)^{\alpha'}\hat{f}_e(\xi,\cdot)}_{L_y^2(0,M)}\le&\ CM^2\norm{\hat{h}_e(\xi,\cdot)}_{L_z^2(0,M)},
			\end{aligned}
		\end{equation}
		for $\alpha'\ge0$.
		
		For the high frequency $|\xi|\ge M^{-3}$, by the asymptotic behavior \eqref{Asym of Ai} of Airy functions,  we have
		\begin{equation}\label{ABy}
			\begin{aligned}
				&|\mu_M(\xi,y)|=\big|Ai((i\xi)^{\frac{1}{3}}y) - \frac{Ai((i\xi)^{\frac{1}{3}}M)}{Bi((i\xi)^{\frac{1}{3}}M)}Bi((i\xi)^{\frac{1}{3}}y)\big|\\
				&\qquad\qquad\ \le C[1+|\xi|^{\frac{1}{3}}y]^{-\frac{1}{4}}e^{-\frac{2}{3}{\rm Re\ }[(i\xi)^{\frac{1}{3}}y]^{\frac{3}{2}}},\\
				&|\mu_0(\xi,y)| = \big|Ai((i\xi)^{\frac{1}{3}}y) + \frac{1}{\sqrt{3}}Bi((i\xi)^{\frac{1}{3}}y)\big|\\
				&\qquad\qquad \le C[1+|\xi|^{\frac{1}{3}}y]^{-\frac{1}{4}}\bigg[e^{-\frac{2}{3}{\rm Re\ }[(i\xi)^{\frac{1}{3}}y]^{\frac{3}{2}}} + e^{\frac{2}{3}{\rm Re\ }[(i\xi)^{\frac{1}{3}}y]^{\frac{3}{2}}}\bigg]\\
				&\qquad\qquad \le C[1+|\xi|^{\frac{1}{3}}y]^{-\frac{1}{4}}e^{\frac{2}{3}{\rm Re\ }[(i\xi)^{\frac{1}{3}}y]^{\frac{3}{2}}},
			\end{aligned}
		\end{equation}
		and then,
		\begin{equation}\label{ABz}
			\begin{aligned}
				&\norm{\mu_0(\xi,z)}_{L_z^2(0,y)}\le C|\xi|^{-\frac{1}{6}}\bigg(\norm{Ai(z)}_{L_z^2(0,(i\xi)^{\frac{1}{3}}y)} + \norm{Bi(z)}_{L_z^2(0,(i\xi)^{\frac{1}{3}}y)}\bigg)\\
				&\qquad\qquad\qquad\quad \le C|\xi|^{-\frac{1}{6}} [1+|\xi|^{\frac{1}{3}}y]^{-\frac{1}{2}}e^{\frac{2}{3}{\rm Re\ }[(i\xi)^{\frac{1}{3}}y]^{\frac{3}{2}}},\\
				&\norm{\mu_M(\xi,z)}_{L_z^2(y,M)}\le C\norm{[1+|\xi|^{\frac{1}{3}}z]^{-\frac{1}{4}}e^{-\frac{2}{3}{\rm Re\ }[(i\xi)^{\frac{1}{3}}z]^{\frac{3}{2}}}}_{L_z^2(y, M)}\\
				&\qquad\qquad\qquad\qquad\le C|\xi|^{-\frac{1}{6}}\norm{[1+\tilde{z}]^{-\frac{1}{4}}e^{-\frac{2}{3}\tilde{z}^{\frac{3}{2}}}}_{L_{\tilde{z}}^2(({\rm Re\ }[(i\xi)^{\frac{1}{3}}y], M))}\\
				&\qquad\qquad\qquad\qquad\le C|\xi|^{-\frac{1}{6}}[1+|\xi|^{\frac{1}{3}}y]^{-\frac{1}{2}}e^{-\frac{2}{3}{\rm Re\ }[(i\xi)^{\frac{1}{3}}y]^{\frac{3}{2}}}.
			\end{aligned}
		\end{equation}
		Combining \eqref{Explicit formula of hatfe}, \eqref{Green func}, \eqref{sM}, \eqref{ABy} with \eqref{ABz}, we obtain
		\begin{equation}\nonumber
			\begin{aligned}
				&|\hat{f}_e(\xi,y)| \le C|s_M|\cdot|\xi|^{-\frac{1}{3}}\bigg\{\big|\mu_M(\xi,y)\big|\cdot\int_0^y \big|\mu_0(\xi,z)\hat{h}_e(\xi,z)\big|dz\\
				&\qquad\qquad\qquad\qquad\qquad\quad+\big|\mu_0(\xi,y)\big|\cdot\int_y^M\big|\mu_M(\xi,z)\hat{h}_e(\xi,z)\big|dz\bigg\}\\
				&\qquad \le C|\xi|^{-\frac{1}{3}}[1+|\xi|^{\frac{1}{3}}y]^{-\frac{1}{4}}\norm{\hat{h}_e(\xi,\cdot)}_{L_z^2([0,M])}\\
				&\qquad\qquad \cdot\bigg\{e^{-\frac{2}{3}{\rm Re\ }[(i\xi)^{\frac{1}{3}}y]^{\frac{3}{2}}}\cdot\norm{\mu_0(\xi,z)}_{L_z^2([0,y])}\\
				&\qquad\qquad\quad+e^{\frac{2}{3}{\rm Re\ }[(i\xi)^{\frac{1}{3}}y]^{\frac{3}{2}}}\cdot\norm{\mu_M(\xi,z)}_{L_z^2([y,M])}\bigg\}\\
				&\qquad \le C|\xi|^{-\frac{1}{2}}[1+|\xi|^{\frac{1}{3}}y]^{-\frac{3}{4}}\norm{\hat{h}_e(\xi,\cdot)}_{L_z^2([0,M])},\quad \mbox{for}\ \ y\in[0,M].
			\end{aligned}
		\end{equation}
		Hence, we obtain a pointwise estimate of $\hat{f}_e$ for the high frequency as follows
		\begin{equation}\label{Pointwise estimate in xi}
			\begin{aligned}
				&\norm{(|\xi|^{\frac{1}{3}}y)^{\alpha'}\hat{f}_e(\xi,\cdot)}_{L_y^2(0,M)}\le C(\alpha')|\xi|^{-\frac{1}{2}}\norm{\hat{h}_e(\xi,\cdot)}_{L_z^2(0,M)}\cdot|\xi|^{-\frac{1}{6}}\\
				&\qquad \le C|\xi|^{-\frac{2}{3}}\norm{\hat{h}_e(\xi,\cdot)}_{L_z^2(0,M)}
				\le CM^2\norm{\hat{h}_e(\xi,\cdot)}_{L_z^2(0,M)},
			\end{aligned}
		\end{equation}
		for $\alpha'\in[0,\frac{1}{4})$. By \eqref{Pointwise estimate in xi0} and \eqref{Pointwise estimate in xi}, we obtain
		\begin{equation}\label{Lower freq estimate}
			\norm{f_e}_{L^2(\Omega_M)} + \norm{|\p_x|^{\frac{1}{18}}y^{\frac{1}{6}}f_e}_{L^2(\Omega_M)}\le CM^2\norm{h_e}_{L^2(\Omega_M)}\le CM^2\norm{h_e}_{L^2(\Omega_M)}.
		\end{equation}
		Here $C$ is independent of $M$ and $f_e$.\smallskip
		
		\textbf{Step 3.} The high-order estimates of $f_e$. \smallskip

		For $\bar{\alpha}\ge0$, we set
		\begin{equation}\label{Def of B0}
			B_0^2 = |\xi|^{2+\bar{\alpha}}\norm{y\hat{f}_e}_{L_{y}^2}^2 + |\xi|^{\frac{4}{3}+\bar{\alpha}}\norm{\hat{f}_e}_{L_{y}^2}^2.
		\end{equation}
		Multiplying the equation \eqref{FT of fe} by $\xi^{1+\bar{\alpha}} y\hat{f}_e$, we obtain
		\begin{equation}\nonumber
			i|\xi|^{2+\bar{\alpha}}\norm{y\hat{f}_e}_{L_{y}^2}^2 +  \xi^{1+\bar{\alpha}}\norm{y^{\frac{1}{2}}\p_{y}\hat{f}_e}_{L_{y}^2}^2 +  \xi^{1+\bar{\alpha}}\langle \p_{y}\hat{f}_e,\hat{f}_e\rangle_{L_{y}^2} = \langle \hat{h}_e, y\xi^{1+\bar{\alpha}}\hat{f}_e\rangle_{L_y^2}.
		\end{equation}
		Here we used the fact that $\hat{f}_e\big|_{y=M} = 0$. 
		Set
		\begin{equation}\label{Def of B1 B2}
			\begin{aligned}
				B_1^2 =&\ |\xi|^{1+\bar{\alpha}}\norm{y^{\frac{1}{2}}\p_{y}\hat{f}_e}_{L_{y}^2}^2,\\
				B_2^2 =&\ |\xi|^{\frac{5}{3}+\bar{\alpha}}\norm{y^{\frac{1}{2}}\hat{f}_e}_{L_{y}^2}^2 +  |\xi|^{\frac{2}{3}+\bar{\alpha}}\norm{\p_{y}\hat{f}_e}_{L_{y}^2}^2.
			\end{aligned}
		\end{equation}
		Then we have
		\begin{equation}\label{Estimate of B1}
			\begin{aligned}
				|\xi|^{2+\bar{\alpha}}\norm{y\hat{f}_e}_{L_{y}^2}^2 +  B_1^2&\le \big|-\xi^{1+\bar{\alpha}}\langle\p_y\hat{f}_e,\hat{f}_e\rangle_{L_y^2} + \langle \hat{h}_e, y\xi^{1+\bar{\alpha}}\hat{f}_e\rangle_{L_y^2}\big|\\
				&\le \big|-\langle\xi^{\frac{1}{3}+\frac{\bar{\alpha}}{2}}\p_y\hat{f}_e,\xi^{\frac{2}{3}+\frac{\bar{\alpha}}{2}}\hat{f}_e\rangle_{L_y^2} + \langle \xi^{\frac{\bar{\alpha}}{2}}\hat{h}_e, y\xi^{1+\frac{\bar{\alpha}}{2}}\hat{f}_e\rangle_{L_y^2}\big|\\
				& \le C\bigg(|\xi|^{\frac{1}{3}+\frac{\bar{\alpha}}{2}}\norm{\p_{y}\hat{f}_e}_{L_{y}^2} + |\xi|^{\frac{\bar{\alpha}}{2}}\norm{\hat{h}_e}_{L_{y}^2}\bigg)B_0\\
				&\le CB_0(B_2 + |\xi|^{\frac{\bar{\alpha}}{2}}\norm{\hat{h}_e}_{L_{y}^2}).
			\end{aligned}
		\end{equation}
		Here $C$ is independent of $\bar{\alpha}$. Multiplying the equation \eqref{FT of fe} by $|\xi|^{\frac{2}{3}+\bar{\alpha}}\hat{f}_e$, we obtain
		\begin{equation}
			\begin{aligned}
				i\xi|\xi|^{\frac{2}{3}+\bar{\alpha}}\norm{y^{\frac{1}{2}}\hat{f}_e}_{L_{y}^2}^2 +  |\xi|^{\frac{2}{3}+\bar{\alpha}}\norm{\p_{y}\hat{f}_e}_{L_{y}^2}^2 = \langle \hat{h}_e,|\xi|^{\frac{2}{3}+\bar{\alpha}}\hat{f}_e\rangle,\quad \mbox{by}\ \ \hat{f}_e(\cdot,M) = 0.
			\end{aligned}
		\end{equation}
		Then we have
		\begin{equation}\nonumber
			\begin{aligned}
				|\xi|^{\frac{5}{3}+\bar{\alpha}}\norm{y^{\frac{1}{2}}\hat{f}_e}_{L_{y}^2}^2 +|\xi|^{\frac{2}{3}+\bar{\alpha}}\norm{\p_{y}\hat{f}_e}_{L_{y}^2}^2&\le \norm{|\xi|^{\frac{2}{3}+\frac{\bar{\alpha}}{2}}\hat{f}_e}_{L_{y}^2}\norm{|\xi|^{\frac{\bar{\alpha}}{2}}\hat{h}_e}_{L_{y}^2}\\
				&\le CB_0\norm{|\xi|^{\frac{\bar{\alpha}}{2}}\hat{h}_e}_{L_{y}^2}.
			\end{aligned}
		\end{equation}
		Thus, we get by \eqref{Def of B1 B2} that
		\begin{equation}\label{Estimate of B2}
			\begin{aligned}
				B_2^2 \le&\ CB_0\norm{|\xi|^{\frac{\bar{\alpha}}{2}}\hat{h}_e}_{L_{y}^2}.
			\end{aligned}
		\end{equation}
		On the other hand, we have
		\begin{equation}\label{Estimate of B01}
			\begin{aligned}
				|\xi|^{\frac{4}{3}+\bar{\alpha}}\norm{\hat{f}_e}_{L_{y}^2}^2&= |\xi|^{\frac{4}{3}+\bar{\alpha}}\int_0^{M}\p_{y}(y)|\hat{f}_e|^2dy\\
				&= |\xi|^{\frac{4}{3}+\bar{\alpha}}M\hat{f}_e^2\big|_{y=M} - |\xi|^{\frac{4}{3}+\bar{\alpha}}\int_0^{M} y(\bar{\hat{f}}_e\p_{y}\hat{f}_e + \hat{f}_e\p_{y}\bar{\hat{f}}_e)dy\\
				&\le 2\norm{|\xi|^{1+\frac{\bar{\alpha}}{2}}y\hat{f}_e}_{L_{y}^2}\norm{|\xi|^{\frac{1}{3}+\frac{\bar{\alpha}}{2}}\p_{y}\hat{f}_e}_{L_{y}^2}\\
				&\le CB_0B_2.
			\end{aligned}
		\end{equation}
		By \eqref{Def of B0}, \eqref{Def of B1 B2}, \eqref{Estimate of B1}, \eqref{Estimate of B2} and \eqref{Estimate of B01}, we have
		\begin{equation}\label{Estimate of B0 B1 final}
			B_0^2 + B_1^2\le C|\xi|^{\bar{\alpha}}\norm{\hat{h}_e}_{L_y^2}^2.
		\end{equation}
		Then by \eqref{Estimate of B2} and \eqref{Estimate of B0 B1 final}, we get
		\begin{equation}\label{Estimate B2 final}
			\begin{aligned}
				B_2^2\le&\ CB_0^2 + C|\xi|^{\bar{\alpha}}\norm{\hat{h}_e}_{L_y^2}^2\le C|\xi|^{\bar{\alpha}}\norm{\hat{h}_e}_{L_y^2}^2.
			\end{aligned}
		\end{equation}
		Applying Plancherel's theorem, we obtain
		\begin{equation}\label{High freq estimate}
			\begin{aligned}
				\norm{f_e}_{Y_{\bar{\alpha}}}^2 \le&\ C\bigg(\norm{|\p_x|^{\frac{\bar{\alpha}}{2}}y\p_xf_e}_{L^2}^2 + \norm{|\p_x|^{\frac{2}{3} + \frac{\bar{\alpha}}{2}}f_e}_{L^2}^2 + \norm{|\p_x|^{\frac{1+\bar{\alpha}}{2}}y^{\frac{1}{2}}\p_yf_e}_{L^2}^2\\
				&\quad + \norm{|\p_x|^{\frac{1}{3}+\frac{\bar{\alpha}}{2}}\p_yf_e}_{L^2}^2 + \norm{|\p_x|^{\frac{\bar{\alpha}}{2}}\p_y^2f_e}_{L^2}^2\bigg)\\ =&\ C\bigg(\norm{|\xi|^{1+\frac{\bar{\alpha}}{2}}y\hat{f}_e}_{L^2}^2 + \norm{|\xi|^{\frac{2}{3} + \frac{\bar{\alpha}}{2}}\hat{f}_e}_{L^2}^2 + \norm{|\xi|^{\frac{1+\bar{\alpha}}{2}}y^{\frac{1}{2}}\p_y\hat{f}_e}_{L^2}^2\\
				&\quad + \norm{|\xi|^{\frac{1}{3}+\frac{\bar{\alpha}}{2}}\p_y\hat{f}_e}_{L^2}^2 + \norm{|\xi|^{\frac{\bar{\alpha}}{2}}\p_y^2\hat{f}}_{L^2}^2\bigg)\\
				\le&\ C\bigg(\int_{\R}B_0^2 + B_1^2 + B_2^2 + B_0^2 + \norm{|\xi|^{\frac{\bar{\alpha}}{2}}\hat{h}_e}_{L_y^2}^2d\xi\bigg)\\
				\le&\ C\norm{|\xi|^{\frac{\bar{\alpha}}{2}}\hat{h}_e}_{L^2}^2\quad\mbox{by}\ \ \eqref{Estimate of B0 B1 final}\  \mbox{and}\ \eqref{Estimate B2 final}\\
				=&\ C\norm{|\p_x|^{\frac{\bar{\alpha}}{2}}h_e}_{L^2}^2.
			\end{aligned}
		\end{equation}
		
		Finally, it follows from \eqref{Lower freq estimate} and \eqref{High freq estimate} that
		\begin{equation}
			\norm{f_e}_{X_0\cap X_{\alpha}}\le C\big[(1+M^{2})\norm{h_e}_{L^2} + \norm{|\p_x|^{\frac{\alpha}{2}}h_e}_{L^2}\big].
		\end{equation}
		Here $C$ is independent of $\alpha,M$, and $f_e$.
	\end{proof}
	
	\begin{rem}\label{rem for sm}
		The Airy functions $Ai(z),Bi(z)$ admit decompositions \cite[Chapter 9]{Olver2010} as follows
		\begin{equation}\nonumber
			\begin{aligned}
				&Ai(z) = \frac{\sqrt{z}}{3}\big(I_{-\frac13}(\frac23 z^{\frac32}) - I_{\frac13}(\frac23 z^{\frac32})\big),\quad z\in\C,\\
				&Bi(z) =\sqrt{\frac z3}\big(I_{-\frac13}(\frac23 z^{\frac32}) + I_{\frac13}(\frac23 z^{\frac32})\big),\quad z\in\C,\\
				&I_{\nu}(z) = e^{-\frac{\nu\pi i}{2}} J_{\nu}(ze^{\frac{\pi i}{2}}),\quad -\pi\le\text{Arg}(z)\le\frac\pi2,\ \nu\in\R.
			\end{aligned}
		\end{equation}
		Here $J_{\nu}(z) = \sum\limits_{k=0}^{\infty}\frac{(-1)^k\big(\frac{z}{2}\big)^{2k+\nu}}{k!\Gamma(k+\nu+1)}$ denotes the Bessel function satisfying the Bessel equation
		$$
		\Big(z^2\frac{d^2}{dz} + z\frac{d}{dz} + (z^2-\nu^2)\Big)J_{\nu}(z)=0.
		$$
		We then derive
		\begin{equation}\label{Series of sm1}
			\begin{aligned}
				\cC(z):=&\ \sqrt{3}Ai(z) + Bi(z) = \sqrt{\frac{4z}{3}}I_{-\frac13}(\frac23 z^{\frac32})\\
				=&\ \sqrt{\frac{4z}{3}}e^{\frac{\pi i}{6}}J_{-\frac13}(\frac23 z^{\frac32}e^{\frac{\pi i}{2}}),\qquad -\frac{2\pi}{3}\le \text{Arg}(z)\le \frac\pi3.
			\end{aligned}
		\end{equation}
		For $\nu>-1$, the Bessel function $J_{\nu}(z)$ has only real zeros in $\C$ \cite[Page 482]{Watson1922}. Then by \eqref{Series of sm1}, the zeros of $\cC(z)$ in the sector $\big\{z\in\C|\ -\frac{2\pi}{3}\le \text{Arg}(z)\le \frac\pi3\big\}$ lie only on rays $\big\{z\in\C|\ \text{Arg}(z) = -\frac{\pi}{3}\ \text{or}\ \frac{\pi}{3}\big\}$. Notice that $\cC(z)$ is entire and $\cC(z) = \cC(ze^{\frac{2\pi i}{3}})$ for any $z\in\C$, then we have that all zeros of $\cC(z)$ in $\C$ belong to $\big\{z\in\C|\ \text{Arg}(z) = \frac{\pi}{3} + \frac{2k\pi}{3},\ k\in\Z\big\}$.
	\end{rem}

	We next solve the following homogeneous problem
	\begin{equation}\label{Prob omegab}
		\left\{\begin{aligned}
			&y\p_x\omega_b - \p_y^2\omega_b = 0,\quad (x,y)\in \Omega_M,\\
			&\p_y\omega_b\big|_{y=0} = \p_x|\p_x|A,\quad x\in\R.
		\end{aligned}\right.
	\end{equation}
	
	\begin{prop}\label{lem for boundary data}
		Let $\bar{\alpha}\ge0$ and $g=\p_x|\p_x|A\in H^{\frac{1}{6}+\frac{\bar{\alpha}}{2}}\cap \dot{H}^{-1}$. Then there exists a solution $\omega_b\in X_{\bar{\alpha}}$ to the boundary-value problem \eqref{Prob omegab} with the estimate
		\begin{equation}
			\norm{\omega_b}_{X_{\bar{\alpha}}}\le C\norm{|\p_x|^{\frac{1}{6}+\frac{\bar{\alpha}}{2}}g}_{L^2(\R)}\le C\norm{|\p_x|^{\frac{5}{6}+\frac{\bar{\alpha}}{2}}A}_{H^{\frac{4}{3}}(\R)},
		\end{equation}
		where $C>0$ is a constant independent of $\bar{\alpha}$ and $A$.
	\end{prop}
	
	\begin{proof}
		To solve the boundary-value problem \eqref{Prob omegab}, we take the Fourier transform in variable $x$. Then we have
		\begin{equation}\label{BVP omegab}
			\left\{\begin{aligned}
				i\xi y\hat{\omega}_b - \p_y^2\hat{\omega}_b =&\ 0,\quad (\xi,y)\in\R\times(0,M)\\
				\p_y\hat{\omega}_b\big|_{y=0} = \hat{g} =&\ i\xi|\xi|\hat{A},\quad \xi\in\R.
			\end{aligned}\right.
		\end{equation}
		The general solution to the BVP \eqref{BVP omegab} is as follows
		\begin{equation}
			\begin{aligned}
				\hat{\omega}_b(\xi,y) = C(\xi)\bigg(\sqrt{3}Ai((i\xi)^{\frac{1}{3}}y) + Bi((i\xi)^{\frac{1}{3}}y)\bigg) + \frac{Ai((i\xi)^{\frac{1}{3}}y)}{Ai'(0)(i\xi)^{\frac{1}{3}}}\hat{g}.
			\end{aligned}
		\end{equation}
		Taking $C(\xi)=0$, we obtain a solution $\omega_b$ with the explicit formula
		\begin{equation}\label{Explicit form of wb}
			\hat{\omega}_b(\xi,y) = \left\{\begin{aligned}
				&\frac{Ai(e^{\frac{\pi}{6}i}\xi^{\frac{1}{3}}y)}{Ai'(0)e^{\frac{\pi}{6}i}\xi^{\frac{1}{3}}}\hat{g}(\xi),\quad \xi>0,\\
				&\frac{Ai(e^{-\frac{\pi}{6}i}(-\xi)^{\frac{1}{3}}y)}{Ai'(0)e^{-\frac{\pi}{6}i}(-\xi)^{\frac{1}{3}}}\hat{g}(\xi),\quad \xi<0.
			\end{aligned}\right.
		\end{equation}
		Then we have
		\begin{equation}\label{Estimate of omegab lower}
			\begin{aligned}
				\norm{(|\p_x|^{\frac{1}{3}}y)^{\alpha'}\omega_b}_{L^2(\Omega_M)}&\le C\norm{(|\xi|^{\frac{1}{3}}y)^{\alpha'}\hat{\omega}_b}_{L^2(\Omega_M)}\\
				& \le C\norm{\norm{Ai((i\xi)^{\frac{1}{3}}y)(|\xi|^{\frac{1}{3}}y)^{\alpha'}}_{L_y^2}\cdot |\xi|^{\frac{5}{3}}|\hat{A}(\xi)|}_{L_{\xi}^2}\\
				& \le C\norm{|\xi|^{\frac{3}{2}}\hat{A}(\xi)}_{L_{\xi}^2}= C\norm{|\p_x|^{\frac{3}{2}}A}_{L^2(\R)}\\
				& \le C\norm{|\p_x|^{\frac{5}{6}}A}_{H^{\frac{2}{3}}(\R)}.
			\end{aligned}
		\end{equation}
		As $g\in H^{\frac{1}{6} + \frac{\alpha}{2}}(\R)$, we obtain
		\begin{equation}\label{Estimate of omegab higher}
			\begin{aligned}
				\norm{\omega_b}_{Y_{\bar{\alpha}}}=&\ \norm{y|\p_x|^{\frac{\bar{\alpha}}{2}}\p_x\omega_b}_{L_{x,y}^2} + \norm{|\p_x|^{\frac{2}{3}+\frac{\bar{\alpha}}{2}}\omega_b}_{L_{x,y}^2} + \norm{|\p_x|^{\frac{1+\bar{\alpha}}{2}}y^{\frac{1}{2}}\p_y\omega_b}_{L_{x,y}^2}\\
				&\ \ + \norm{|\p_x|^{\frac{1}{3}+\frac{\bar{\alpha}}{2}}\p_y\omega_b}_{L_{x,y}^2} + \norm{|\p_x|^{\frac{\bar{\alpha}}{2}}\p_{y}^2\omega_b}_{L_{x,y}^2}\\
				\le&\ C\norm{|\p_x|^{\frac{1}{6}+\frac{\bar{\alpha}}{2}}g}_{L^2(\R)}\le C\norm{|\p_x|^{\frac{13+3\bar{\alpha}}{6}}A}_{L^2(\R)}.
			\end{aligned}
		\end{equation}
		It follows  from \eqref{Estimate of omegab lower} and \eqref{Estimate of omegab higher} that $\omega_b\in X_{\bar{\alpha}}$ with
		\begin{equation}\nonumber
			\norm{\omega_b}_{X_{\bar{\alpha}}}\le C\norm{|\p_x|^{\frac{5}{6}+\frac{\bar{\alpha}}{2}}A}_{H^{\frac{4}{3}}(\R)}.
		\end{equation}
	\end{proof}
	
	Now we are in a position to address the system \eqref{Eq for omega0-1}, i.e.,  complete the proof of Theorem \ref{Main thm0-0}. Let $(\omega_0,u_0,v_0,A_0)$ be the solution to the linear system
	\begin{equation}\label{Eq of omega0}
		\left\{\begin{aligned}
			&y\p_x\omega_0 - \p_{y}^2\omega_0 = 0,\\
			&u_0 = I_{y}[\omega_0],\quad v_0 = -\p_xI_{y}[u_0],\\
			&\p_y\omega_0\big|_{y=0} = \p_x|\p_x|A_0(x),\quad I_{M}[\omega_0] = A_0(x) + F(x).
		\end{aligned}\right.
	\end{equation}
	Unless otherwise specified, the parameters $\alpha,\bar{\alpha}$ satisfy $\alpha>0$ and $\bar{\alpha}\ge0$.


	\begin{proof}[Proof of Theorem \ref{Main thm0-0}]
		We take \eqref{Explicit form of wb} as a solution to \eqref{Eq of omega0}, that is,
		\begin{equation}\label{Formula0}
			\hat{\omega}_0(\xi,y) = \frac{Ai((i\xi)^{\frac{1}{3}}y)}{Ai'(0)(i\xi)^{\frac{1}{3}}}(i\xi|\xi|\hat{A}_0) = \left\{\begin{aligned}
				&\frac{Ai(e^{\frac{\pi}{6}i}\xi^{\frac{1}{3}}y)}{Ai'(0)e^{\frac{\pi}{6}i}}i\xi|\xi|^{\frac{2}{3}}\hat{A}_0(\xi),\quad \xi>0,\\
				&\frac{Ai(e^{-\frac{\pi}{6}i}(-\xi)^{\frac{1}{3}}y)}{Ai'(0)e^{-\frac{\pi}{6}i}}i\xi|\xi|^{\frac{2}{3}}\hat{A}_0(\xi),\quad \xi<0.
			\end{aligned}\right. 
		\end{equation}
		Here $\hat{f}(\xi,y) = \cF_{x\to\xi}[f](\xi,y)$. It is straight to verify that
		\begin{equation}
			\omega_0(\cdot,y)\big|_{y=M} = \frac{Ai(M\p_x^{\frac{1}{3}})}{Ai'(0)}\p_x^{\frac{2}{3}}|\p_x|A_0.
		\end{equation}
		Substituting the formula \eqref{Formula0} into the integral condition $I_M[\hat{\omega}_0] = \hat{A}_0 + \hat{F}$, we have
		\begin{equation}\label{Formula1}
			\begin{aligned}
				\hat{A}_0 + \hat{F} =&\ \int_0^M\hat{\omega}_0(\xi,y) dy\\
				=&\ \int_0^M Ai(e^{\frac{\sgn\ \xi}{6}\pi i}|\xi|^{\frac{1}{3}}y)dy \cdot\frac{1}{Ai'(0)e^{\frac{\sgn\ \xi}{6}\pi i}}i\xi|\xi|^{\frac{2}{3}}\hat{A}_0\\
				=&\ \int_0^{e^{\frac{\sgn\ \xi}{6}\pi i}|\xi|^{\frac{1}{3}}M}Ai(z)dz\cdot\frac{1}{Ai'(0)e^{\frac{\sgn\ \xi}{3}\pi i}}i\xi|\xi|^{\frac{1}{3}}\hat{A}_0\\
				=&\ \frac{p_M(\xi)}{Ai'(0)e^{\frac{\sgn\ \xi}{3}\pi i}}i\xi|\xi|^{\frac{1}{3}}\hat{A}_0.
			\end{aligned}
		\end{equation}
		Here
		\begin{equation}\nonumber
			\begin{aligned}
				p_M(\xi) =&\ \int_0^{(i\xi)^{\frac{1}{3}}M}Ai(z)dz.
			\end{aligned}
		\end{equation}
		Indeed, the integral is well-defined, since $Ai(z)$ is analytic.  Let
		\begin{equation}\label{m-multiplier}
			\begin{aligned}
				m(\xi) =&\ 1 - \frac{p_M(\xi)}{Ai'(0)e^{\frac{{\rm sgn\ }\xi}{3}\pi i}}i\xi|\xi|^{\frac{1}{3}}.
			\end{aligned}
		\end{equation}
		Then we get by \eqref{Formula1} that
		\begin{equation}\label{FormulaA0}
			\hat{A}_0 = -\frac{1}{m(\xi)}\hat{F}.
		\end{equation}
		It follows from \eqref{Formula0} that
		\begin{equation}\label{Formulaomega0}
			\begin{aligned}
				\hat{\omega}_0(\xi,y) = -\frac{Ai(e^{\frac{{\rm sgn\ }\xi}{6}\pi i}\xi^{\frac{1}{3}}y)}{Ai'(0)e^{\frac{{\rm sgn\ }\xi}{6}\pi i}}\frac{i\xi|\xi|^{\frac{2}{3}}}{m(\xi)}\hat{F}(\xi).
			\end{aligned}
		\end{equation}
		
		To ensure that the formulas \eqref{FormulaA0} and \eqref{Formulaomega0} are well-defined, we need to prove $m(\xi)\not=0$ for any $\xi\in\R$. We define
		$$
		p_{0}:=\mathop{sup}\limits_{{\rm Arg\ }z\in\{\pm\frac{\pi}{6}\}}\bigg|\int_0^{z}Ai(\zeta)d\zeta\bigg|.
		$$
		Noting that $0\le p_0<\infty$, we obtain
		$$
		a_0:=\bigg(\frac{4p_0}{3|Ai'(0)|}\bigg)^{-\frac{3}{4}}>0.
		$$
		It follows that
		\begin{equation}\label{Multiplier m 1}
			m(\xi)\ge \frac{1}{4}\quad \forall\ |\xi|\le a_0.
		\end{equation}
		Using the following identity
		\begin{equation}
			\begin{aligned}
				\int_0^{\infty}Ai(z)dz = \frac{1}{3},
			\end{aligned}
		\end{equation}
		the multiplier $m(\xi)$ converges to $\frac{1}{3}$ uniformly on $[a_0,\infty)$  as $M\to\infty$. A direct calculation yields 
		\begin{equation}\label{Multiplier m 2}
			\bigg|1 - \frac{1}{3Ai'(0)e^{\frac{{\rm sgn\ }\xi}{3}\pi i}}i\xi|\xi|^{\frac{1}{3}}\bigg|^{-1}\le C(1+|\xi|)^{-\frac{4}{3}},\quad \forall\ \xi\in\R.
		\end{equation}
		Then there exists a large enough constant $M_0>0$ such that for $M\ge M_0$, we get by \eqref{Multiplier m 1} and \eqref{Multiplier m 2} that
		\begin{equation}
			|m(\xi)|^{-1} = \bigg|1 - \frac{p_{M}(\xi)}{Ai'(0)e^{\frac{{\rm sgn\ }\xi}{3}\pi i}}i\xi|\xi|^{\frac{1}{3}}\bigg|^{-1}\le C(1+|\xi|)^{-\frac{4}{3}},\quad \forall\ \xi\in\R.
		\end{equation}
		Thus, $\hat{A}_0$ and $\hat{\omega}_0$ are well-defined. Moreover,  there exists a constant $C>0$ such that 
		\begin{equation}\label{Estimate0}
			\norm{m(D)^{-1}f}_{H^{\frac{4}{3}}(\R)}\le C\norm{f}_{L^2(\R)}.
		\end{equation}
		
		Using Proposition \ref{lem for boundary data} and \eqref{Estimate0}, for any $\bar{\alpha}\ge0$, we obtain the $X_{\bar{\alpha}}$-estimate of $\omega_0$ as follows 
		\begin{equation}\nonumber
			\begin{aligned}
				\norm{\omega_0}_{X_{\bar{\alpha}}}\le C\norm{|\p_x|^{\frac{5}{6}+\frac{\bar{\alpha}}{2}}A_0}_{H^{\frac{4}{3}}(\R)}\le C\norm{|\p_x|^{\frac{5}{6}+\frac{\bar{\alpha}}{2}}F}_{L^2(\R)}\le C\norm{|\p_x|^{\frac{5}{6}}F}_{\dot{H}^{\frac{\bar{\alpha}}{2}}(\R)}.
			\end{aligned}
		\end{equation}

		For the bound of displacement $A_0$ and the velocity $u_0(x)$, we have
		\begin{equation}\nonumber
			\begin{aligned}
				&\norm{u_0}_{L^{\infty}(\Omega_M)} + \norm{A_0}_{L^{\infty}(\R)}\le \norm{\omega_0}_{L_x^{\infty}L_y^1}\le \norm{\omega_0}_{L_y^1L_x^{\infty}}\le \norm{\omega_0}_{X_0}.
			\end{aligned}
		\end{equation}
		
		This completes the proof.
	\end{proof}
	
	\begin{rem}\label{rmk: M}
		Note that ${\rm Re\ }\bigg(\frac{p_M(\xi)}{e^{\frac{{\rm sgn\ }\xi}{3}\pi i}}\bigg)> 0$ for $|\xi|^{\frac{1}{3}}M>0$ by numberical results (see figure \ref{Analysis about m}, where $r=|\xi|^{\frac{1}{3}}M$). Then by the continuity of $m(\xi)$, we have $$\frac{1}{|m(\xi)|}\le C(1+|\xi|)^{-\frac{4}{3}}\text{ for any } \xi\in\R\text{ and }M>0,$$ where $C>0$ depends on $M^{-1}$.
		\begin{figure}[h]
			\centering
			\includegraphics[scale=0.2]{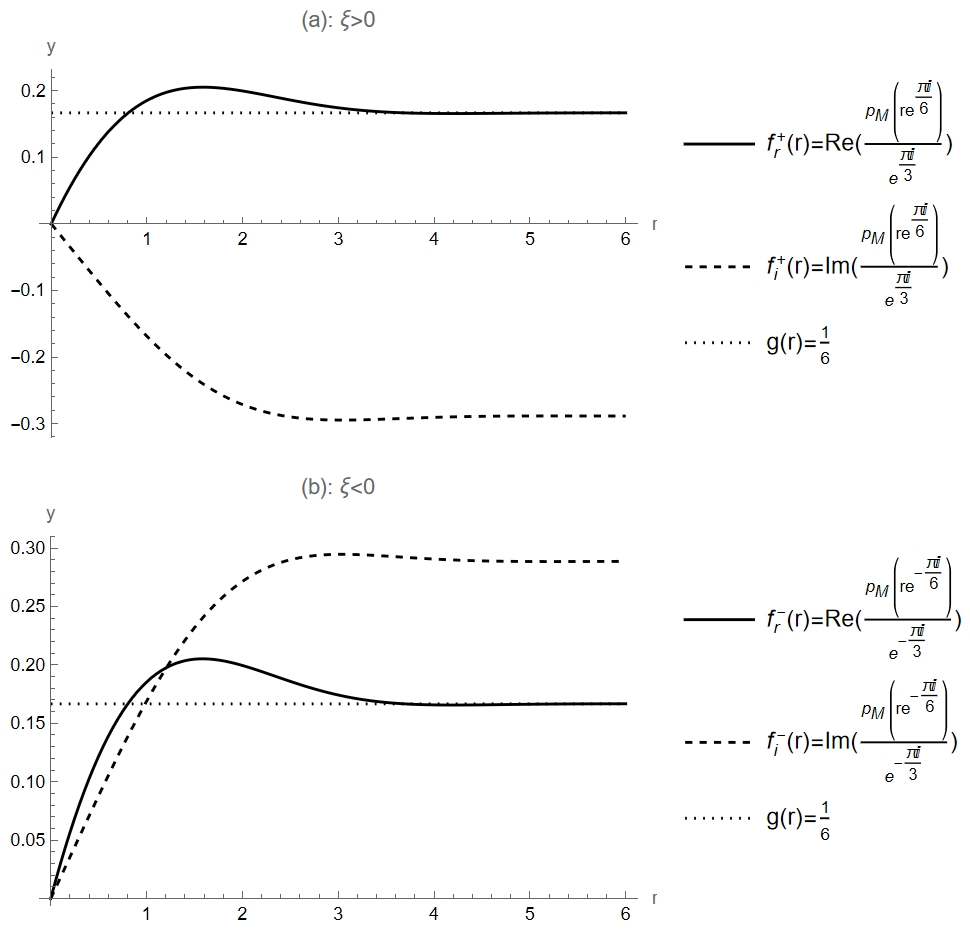}
			\caption{Analysis on $m(\xi)$.}
			\label{Analysis about m}
		\end{figure}
	\end{rem}
	
	
	\section{Existence and uniqueness of classical solution}
	
	In this section, we establish the existence and uniqueness of classical solution to the nonlinear system \eqref{Perturbation system0}, i.e., 
	\begin{equation}\label{Perturbation system00}
		\left\{\begin{aligned}
			&y\p_x\omega - \p_y^2\omega = -u\p_x\omega - v\p_y\omega,\\
			&u = I_{y}[\omega],\quad v = -\p_xI_{y}[u],\\
			&\p_y\omega\big|_{y=0} = \p_x|\p_x|A,\quad I_{M}[\omega] = A(x) + F(x).
		\end{aligned}\right.
	\end{equation}
	To construct a solution to the system \eqref{Perturbation system00}, we use the iteration method. Let
	\begin{align*} (\bar{\omega},\bar{A},\bar{u},\bar{v}) = (\omega - \omega_0, A - A_0, u - u_0, v-v_0).
	\end{align*}
	Then the system \eqref{Perturbation system00} is reduced to the following BVP:
	\begin{equation}\label{System of baromega0}
		\left\{\begin{aligned}
			&y\p_x\bar{\omega} - \p_y^2\bar{\omega} = -(\bar{u} + u_0)\p_x(\bar{\omega} + \omega_0) - (\bar{v} + v_0)\p_y(\bar{\omega} + \omega_0),\\
			&\bar{u} = I_y[\bar{\omega}],\quad \bar{v} = -\p_xI_y[\bar{u}],\\
			&\p_y\bar{\omega}\big|_{y=0} = \p_x|\p_x|\bar{A},\quad I_{M}[\bar{\omega}]=\bar{A}.
		\end{aligned}\right.
	\end{equation}
	We introduce the iterated system as follows
	\begin{equation}\label{System of baromega}
		\left\{\begin{aligned}
			&y\p_x\bar{\omega}_{k+1} - \p_y^2\bar{\omega}_{k+1} = -(\bar{u}_k + u_0)\p_x(\bar{\omega}_k + \omega_0) - (\bar{v}_k + v_0)\p_y(\bar{\omega}_k + \omega_0),\\
			&\bar{u}_{k+1} = I_y[\bar{\omega}_{k+1}],\quad \bar{v}_{k+1} = -\p_xI_y[\bar{u}_{k+1}],\\
			&\bar{u}_{k} = I_y[\bar{\omega}_{k}],\quad \bar{v}_{k} = -\p_xI_y[\bar{u}_{k}],\\
			&\p_y\bar{\omega}_{k+1}\big|_{y=0} = \p_x|\p_x|\bar{A}_{k+1},\quad I_{M}[\bar{\omega}_{k+1}]=\bar{A}_{k+1}.
		\end{aligned}\right.
	\end{equation}

	First of all, we present the following uniform estimate for the iterated system \eqref{System of baromega}.
	
	\begin{prop}\label{prop interation}
		Given $\alpha\in(\frac23,\frac73]$, assume that $(\bar{\omega}_k,\bar{A}_k)\in(X_0\cap X_{\alpha})\times  X_{\alpha,\infty}$. Then there exists a solution $(\bar{\omega}_{k+1},\bar{A}_{k+1})\in (X_0\cap X_{\alpha})\times  X_{\alpha,\infty}$ to \eqref{System of baromega}, which holds that for $\e\in(0,\frac16]$, 
		\begin{equation}
			\begin{aligned}
				&\norm{\bar{\omega}_{k+1}}_{X_0\cap X_{\alpha}} + \norm{\bar{A}_{k+1}}_{ X_{\alpha,\infty}}\\
				&\le CM^{\frac{5+3\e}{2}}\Big[\big(\norm{\bar{\omega}_k}_{X_0\cap X_{\alpha}} + \norm{\bar{A}_{k}}_{X_{\alpha,\infty}}\big)^2 + \norm{F}_{H^2(\R)}^2\Big].
			\end{aligned}
		\end{equation}
		Here $C=C(\alpha,\e)>0$ is a constant.
	\end{prop}
	
	\begin{proof}
		We decompose $\bar{\omega}_{k+1} = \bar{\omega}_{e,k+1} + \bar{\omega}_{b,k+1}$, where $\bar{\omega}_{e,k+1}$ and $\bar{\omega}_{b,k+1}$ satisfy
		\begin{equation}\nonumber
			\left\{\begin{aligned}
				&y\p_x\bar{\omega}_{e,k+1} - \p_y^2\bar{\omega}_{e,k+1} = \bar{h}_k,\\
				&\p_y\bar{\omega}_{e,k+1}\big|_{y=0}= 0,\quad \bar{\omega}_{e,k+1}\big|_{y=M} = 0,
			\end{aligned}\right.
		\end{equation}
		and 
		\begin{equation}\nonumber
			\left\{\begin{aligned}
				&y\p_x\bar{\omega}_{b,k+1} - \p_y^2\bar{\omega}_{b,k+1} = 0,\\
				&\p_y\bar{\omega}_{b,k+1}\big|_{y=0}= \p_x|\p_x|\bar{A}_{k+1},
			\end{aligned}\right.
		\end{equation}
		together with
		\begin{equation}\label{Eq of barAk+1}
			I_M[\bar{\omega}_{e,k+1} + \bar{\omega}_{b,k+1}] = \bar{A}_{k+1}.
		\end{equation}
		Here
		$$\bar{h}_k = -(\bar{u}_k+u_0)\p_x(\bar{\omega}_k + \omega_0) - (\bar{v}_k + v_0)\p_y(\bar{\omega}_k + \omega_0).$$        
		
		By Proposition \ref{lem for nonlinear terms}, we obtain the unique solution $\bar{\omega}_{e,k+1}$ with the form
		\begin{equation}\label{Formala of we,k+1}
			\begin{aligned}
				\bar{\omega}_{e,k+1} =&\  \pi s_M(i\xi)^{-\frac{1}{3}}\bigg\{\mu_{M}(\xi,y)\int_0^y\mu_{0}(\xi,z)\hat{\bar{h}}_k(\xi,z)dz\\
				&\qquad\qquad\qquad +\mu_{0}(\xi,y)\int_y^M \mu_{M}(\xi,z)\hat{\bar{h}}_k(\xi,z)dz\bigg\}.
			\end{aligned}
		\end{equation}
		By Proposition \ref{lem for boundary data}, we can find a solution $\bar{\omega}_{b,k+1}$ with
		\begin{equation}\label{Formala of wb,k+1}
			\hat{\omega}_{b,k+1}(\xi,y) = \left\{\begin{aligned}
				&\frac{Ai(e^{\frac{\pi}{6}i}\xi^{\frac{1}{3}}y)}{Ai'(0)e^{\frac{\pi}{6}i}\xi^{\frac{1}{3}}}i\xi|\xi|\hat{\bar{A}}_{k+1}(\xi),& \xi>0,\\
				&\frac{Ai(e^{-\frac{\pi}{6}i}(-\xi)^{\frac{1}{3}}y)}{Ai'(0)e^{-\frac{\pi}{6}i}(-\xi)^{\frac{1}{3}}}i\xi|\xi|\hat{\bar{A}}_{k+1}(\xi),&xi<0.
			\end{aligned}\right.
		\end{equation}
		In view of \eqref{Eq of barAk+1}, we have
		\begin{equation}\label{Form of Ak+1}
			\bar{A}_{k+1} = m(D)^{-1}I_{M}[\bar{\omega}_{e,k+1}].
		\end{equation}
		Here $m=m(\xi)$ is defined in \eqref{m-multiplier}.
		
		By Proposition \ref{lem for nonlinear terms}, Theorem \ref{Main thm0-0} and Proposition \ref{prop for forced terms}, we obtain
		\begin{equation}\label{Est of omega e,k+1}
			\begin{aligned}
				\norm{\bar{\omega}_{e,k+1}}_{X_0\cap X_{\alpha}}\le&\ C\big((1+M^2)\norm{\bar{h}_k}_{L^2(\Omega_M)} + \norm{|\p_x|^{\frac{\alpha}{2}}\bar{h}_k}_{L^2(\Omega_M)}\big)\\
				\le&\ CM^{\frac{5+3\e}{2}}\big(\norm{\bar{\omega}_0}_{X_0\cap X_{\alpha}}^2 + \norm{\bar{\omega}_k}_{X_0\cap X_{\alpha}}^2 + \norm{\p_xA_0}_{L^{\infty}\cap \dot{H}^{\frac{\alpha}{2}}}^2\\
				&\qquad + \norm{\p_xA_k}_{L^{\infty}\cap \dot{H}^{\frac{\alpha}{2}}}^2 + \norm{\p_xF}_{L^{\infty}\cap \dot{H}^{\frac{\alpha}{2}}}^2\big)\\
				\le&\ CM^{\frac{5+3\e}{2}}\big(\norm{\bar{\omega}_k}_{X_0\cap X_{\alpha}}^2 + \norm{A_k}_{ X_{\alpha,\infty}}^2 + \norm{F}_{H^2(\R)}^2\big).
			\end{aligned}
		\end{equation}
		By Proposition \ref{lem for boundary data}, Lemma \ref{prop for IMf} and \eqref{Form of Ak+1}, we have
		\begin{equation}\label{Est of Ak+1}
			\begin{split}
				\norm{\bar{A}_{k+1}}_{X_{\alpha,\infty}}&= \norm{|\p_x|^{\frac{5}{6}}\bar{A}_{k+1}}_{H^{\frac{4}{3}}(\R)} + \norm{|\p_x|^{\frac{5}{6}+\frac{\alpha}{2}}\bar{A}_{k+1}}_{H^{\frac{4}{3}}(\R)}\\
				&\le C\norm{|\p_x|^{\frac{5}{6}}I_M[\bar{\omega}_{e,k+1}]}_{L^2(\R)} + C\norm{|\p_x|^{\frac{5}{6}+\frac{\alpha}{2}}I[\bar{\omega}_{e,k+1}]}_{L^2(\R)}\\
				&\le C\norm{\bar{\omega}_{e,k+1}}_{X_0\cap X_{\alpha}},\\
				\norm{\bar{\omega}_{b,k+1}}_{X_0\cap X_{\alpha}}&\le C\norm{|\p_x|^{\frac{5}{6}}\bar{A}_{k+1}}_{H^{\frac{4}{3}}(\R)} + C\norm{|\p_x|^{\frac{5}{6}+\frac{\alpha}{2}}\bar{A}_{k+1}}_{H^{\frac{4}{3}}(\R)}\\
				&\le C\norm{\bar{\omega}_{e,k+1}}_{X_0\cap X_{\alpha}},
			\end{split}
		\end{equation}
		from which and  \eqref{Est of omega e,k+1}, we arrive at
		\begin{equation}\nonumber
			\begin{aligned}
				&\norm{\bar{\omega}_{k+1}}_{X_0\cap X_{\alpha}} + \norm{\bar{A}_{k+1}}_{ X_{\alpha,\infty}}\\
				&\le \norm{\bar{\omega}_{e,k+1}}_{X_0\cap X_{\alpha}} + \norm{\bar{\omega}_{b,k+1}}_{X_0\cap X_{\alpha}} + \norm{\bar{A}_{k+1}}_{X_0\cap X_{\alpha}}\\
				& \le C\norm{\bar{\omega}_{e,k+1}}_{X_0\cap X_{\alpha}}\\
				&\le CM^{\frac{5+3\e}{2}}\Big[\big(\norm{\bar{\omega}_k}_{X_0\cap X_{\alpha}}^2 + \norm{A_k}_{ X_{\alpha,\infty}}^2\big) + \norm{F}_{H^2(\R)}^2\Big].
			\end{aligned}
		\end{equation}
		
		The proof is completed.
	\end{proof}

	To obtain the uniform estimates of the solution sequence,  we use the following basic lemma, which can be proved by the induction.
	
	\begin{lem}\label{lem for iteration}
		Let $\alpha_i,\beta_i\ (i=1,2)$ be positive constants. Assume that $\{a_k\}_{k\ge0},\ \{b_{k}\}_{k\ge1}$ are two non-negative sequences satisfying
		\begin{equation}\nonumber
			a_{k+1}\le \alpha_1 a_k^2 + \alpha_2\quad \mbox{for any }k\ge0;\quad  b_{k+1}\le \beta_1 b_k + \beta_2b_k^2\quad \mbox{for any }k\ge1.
		\end{equation}
		Assume that $\alpha_2\le\frac{1}{4\alpha_1}$ and $\beta_1<1$. Set $a_0=0$ and $b_1\le\frac{\gamma-\beta_1}{\beta_2}$ for some $\gamma\in(\beta_1,1)$. Then we have 
		\begin{equation}\nonumber
			0\le \sup_{k\ge0} a_k\le 2\alpha_2,\qquad \sum_{k=1}^{\infty}b_k \le\frac{b_1}{1-\gamma}.
		\end{equation}
	\end{lem}

	With the above preparations, we complete the proof of Theorem \ref{Main thm0}.
	
	\begin{proof}[Proof of Theorem \ref{Main thm0}]
		The proof consists of three steps.\smallskip
		
		\textbf{Step 1.} Boundedness of the solution sequence.\smallskip
		
		Taking $\bar{\omega}_0 = \bar{A}_0=0$, by Proposition \ref{prop interation}, we obtain a sequence $\{(\bar{\omega}_k,\bar{A}_k)\}_{k\ge0}$ in $(X_0\cap X_{\alpha})\times  X_{\alpha,\infty}$.
		Set $a_k = \norm{\bar{\omega}_k}_{X_0\cap X_{\alpha}} + \norm{\bar{A}_{k}}_{ X_{\alpha,\infty}}$. Clearly, we have $a_0 = 0$. By Proposition \ref{prop interation}, we have
		\begin{equation}\nonumber
			a_{k+1}\le CM^{\frac{5+3\e}{2}}a_k^2 +CM^{\frac{5+3\e}{2}}\norm{F}_{H^2(\R)}^2.
		\end{equation}
		By Lemma \ref{lem for iteration}, for $4C^2M^{5+3\e}\norm{F}_{H^2(\R)}^2\le1$ (that is, $\norm{F}_{H^2(\R)}\le \frac{1}{2CM^{\frac{5+3\e}{2}}}$), we obtain
		\begin{equation}\nonumber
			\sup_{k\ge0}\big(\norm{\bar{\omega}_k}_{X_0\cap X_{\alpha}} + \norm{\bar{A}_{k}}_{ X_{\alpha,\infty}}\big) = \sup_{k\ge0}a_k\le 2CM^{\frac{5+3\e}{2}}\norm{F}_{H^2(\R)}^2.
		\end{equation}
		
		\textbf{Step 2.} Convergence of the solution sequence. \smallskip
		
		Set
		$$\begin{aligned}
			\rw_k =&\ \bar{\omega}_{k} - \bar{\omega}_{k-1}\qquad \mbox{for}\ \ k\ge1,\\
			\ru_k =&\ \bar{u}_{k} - \bar{u}_{k-1} = I_y[\rw_k]\qquad \mbox{for}\ \ k\ge1,\\
			\rv_k =&\ \bar{v}_k - \bar{v}_{k-1} = -\p_xI_y[\ru_k]\qquad \mbox{for}\ \ k\ge1,\\
			\rA_k =&\ \bar{A}_{k} - \bar{A}_{k-1}\qquad \mbox{for}\ \ k\ge1.
		\end{aligned}$$
		By subtracting the equation \eqref{System of baromega} corresponding to $\bar{\omega}_{k+1}$ from the equation \eqref{System of baromega} corresponding to $\bar{\omega}_{k}$, we obtain
		\begin{equation}\nonumber
			\left\{\begin{aligned}
				&y\p_x\rw_{k+1} - \p_y^2\rw_{k+1} = -\big[(\bar{u}_k + u_0)\p_x\rw_k + (\bar{v}_k + v_0)\p_y\rw_k\\
				&\qquad\qquad\qquad\qquad\qquad\  + \ru_k\p_x(\bar{\omega}_k + \omega_0) - \rv_k\p_y(\bar{\omega}_k + \omega_0)\big],\\
				&\ru_{k+1} = I_y[\rw_{k+1}],\quad \rv_{k+1} = -\p_xI_y[\ru_{k+1}],\\
				&\ru_{k} = I_y[\rw_{k}],\quad \rv_{k} = -\p_xI_y[\ru_{k}],\\
				&\p_y\rw_{k+1}\big|_{y=0} = \p_x|\p_x|\rA_{k+1},\quad I_{M}[\rw_{k+1}]=\rA_{k+1}.
			\end{aligned}\right.
		\end{equation}
		We  denote
		$$\begin{aligned}
			&\rh_k = \bar{h}_{k} - \bar{h}_{k-1} = -(\bar{u}_k + u_0)\p_x\rw_k - (\bar{v}_k + v_0)\p_y\rw_k - \ru_k\p_x(\bar{\omega}_k + \omega_0) - \rv_k\p_y(\bar{\omega}_k + \omega_0),\\
			&\rw_{e,k} = \bar{\omega}_{e,k}-\bar{\omega}_{e,k-1},\quad \rw_{b,k} = \bar{\omega}_{b,k}-\bar{\omega}_{b,k-1}.
		\end{aligned}$$
		Then by the explicit formulas \eqref{Formala of we,k+1}, \eqref{Formala of wb,k+1} and \eqref{Form of Ak+1}, we can also obtain the estimates \eqref{Est of omega e,k+1} and \eqref{Est of Ak+1} for $(\rw_{e,k+1},\rA_{k+1}, \rw_{b,k+1})$ as follows 
		\begin{equation}\label{Est of rw,rAk+1}
			\begin{aligned}
				&\norm{\rw_{e,k+1}}_{X_0\cap X_{\alpha}}\le C\big((1+M^2)\norm{\rh_k}_{L^2(\Omega_M)} + \norm{|\p_x|^{\frac{\alpha}{2}}\rh_k}_{L^2(\Omega_M)}\big),\\
				&\norm{\rA_{k+1}}_{X_{\alpha,\infty}}\le C\norm{\rw_{e,k+1}}_{X_0\cap X_{\alpha}},\\
				&\norm{\rw_{b,k+1}}_{X_0\cap X_{\alpha}}\le C\norm{\rw_{e,k+1}}_{X_0\cap X_{\alpha}}.
			\end{aligned}
		\end{equation}
		
		In view of Remark \ref{rem for forced terms}, we have
		\begin{equation}\label{Est of rhk}
			\begin{aligned}
				&\norm{\rh_k}_{L^2}\le CM^{\frac{1+3\e}{2}}\big(\norm{\rw_k}_{X_0\cap X_{\alpha}} + \norm{\rA_k}_{X_{\alpha,\infty}}\big)\norm{\omega_0}_{X_0\cap X_{\alpha}},\\
				&\norm{|\p_x|^{\frac{\alpha}{2}}\rh_k}_{L^2}\le CM\norm{\rw_k}_{X_0\cap X_{\alpha}}^2 + CM\big(\norm{\rw_k}_{X_0\cap X_{\alpha}} + \norm{\rA_k}_{X_{\alpha,\infty}}\big)\norm{\omega_0}_{X_0\cap X_{\alpha}}\\
				&\qquad \le CM\big(\norm{\rw_k}_{X_0\cap X_{\alpha}} + \norm{F}_{H^2(\R)}\big)\big(\norm{\rw_k}_{X_0\cap X_{\alpha}} + \norm{\rA_k}_{X_{\alpha,\infty}}\big).
			\end{aligned}
		\end{equation}
		It follows from \eqref{Est of rw,rAk+1} and \eqref{Est of rhk} that
		\begin{equation}\nonumber
			\begin{aligned}
				&\norm{\rw_{k+1}}_{X_0\cap X_{\alpha}} + \norm{\rA_{k+1}}_{ X_{\alpha,\infty}}\\
				&\le CM^{\frac{5+3\e}{2}}\big(\norm{\rw_{k}}_{X_0\cap X_{\alpha}} + \norm{\rA_{k}}_{ X_{\alpha,\infty}}\big)\big(\norm{\rw_{k}}_{X_0\cap X_{\alpha}} + \norm{\rA_{k}}_{ X_{\alpha,\infty}} + \norm{F}_{H^2(\R)}\big)\\
				&\le CM^{\frac{5+3\e}{2}}\Big[\norm{F}_{H^2(\R)}(\norm{\rw_{k}}_{X_0\cap X_{\alpha}} + \norm{\rA_{k}}_{ X_{\alpha,\infty}})
				+\big(\norm{\rw_{k}}_{X_0\cap X_{\alpha}} + \norm{\rA_{k}}_{ X_{\alpha,\infty}}\big)^2\Big].
			\end{aligned}
		\end{equation}
		By Lemma \ref{lem for iteration}, if $CM^{\frac{5+3\e}{2}}\norm{F}_{H^2(\R)}\le\frac{1}{2}$ and
		$$
		\norm{\rw_{1}}_{X_0\cap X_{\alpha}} + \norm{\rA_{1}}_{ X_{\alpha,\infty}}\le 2CM^{\frac{5+3\e}{2}}\norm{F}_{H^2(\R)}^2\le\frac{\frac{3}{4}-CM^{\frac{5+3\e}{2}}\norm{F}_{H^2(\R)}}{CM^{\frac{5+3\e}{2}}\norm{F}_{H^2(\R)}},
		$$
		which means that $\norm{F}_{H^2(\R)}\le \frac{1}{2CM^{\frac{5+3\e}{2}}}$ for $M\gg1$, then there holds
		$$\begin{aligned}
			\sum_{k\ge1}\big(\norm{\rw_{k}}_{X_0\cap X_{\alpha}} + \norm{\rA_{k}}_{ X_{\alpha,\infty}}\big)<\infty.
		\end{aligned}$$
		Since $\bar{\omega}_k = \bar{\omega}_0 + \sum_{j=1}^{k}\rw_j,\ \bar{A}_k = \bar{A}_0 + \sum_{j=1}^{k}\rA_j$, we conclude that $\{(\bar{\omega}_k, \bar{A}_k)\}_{k\ge0}$ converges to some pair $(\bar{\omega},\bar{A})$ in $(X_0\cap X_{\alpha})\times  X_{\alpha,\infty}$. \smallskip
		
		\textbf{Step 3}. Uniqueness of the solution $(\omega_0 + \bar{\omega},A_0 + \bar{A})$ in $\cX_{\alpha,M}$.\smallskip
		
		By Theorem \ref{Main thm0-0}, we know $(\omega_0,A_0)\in\cX_{\alpha,M}$. Combing \eqref{Formala of we,k+1}-\eqref{Formala of wb,k+1} with the convergence of $\{(\bar{\omega}_{k},\bar{A}_{k})\}_{k\ge0}$ in $X_0\cap X_{\alpha}\times X_{\alpha,\infty}$, we conclude that $(\bar{\omega},\bar{A})\in (X_0\cap X_{\alpha})\times X_{\alpha,\infty}$ and 
		\begin{equation}\nonumber
			\begin{aligned}
				\hat{\bar{\omega}}(\xi,y)\big|_{y=M}& = \hat{\bar{\omega}}_{b}(\xi,y)\big|_{y=M}\\
				&= \left\{\begin{aligned}
					&\frac{Ai(e^{\frac{\pi}{6}i}\xi^{\frac{1}{3}}M)}{Ai'(0)e^{\frac{\pi}{6}i}\xi^{\frac{1}{3}}}i\xi|\xi|\hat{\bar{A}}_{k+1}(\xi),\quad \xi>0,\\
					&\frac{Ai(e^{-\frac{\pi}{6}i}(-\xi)^{\frac{1}{3}}M)}{Ai'(0)e^{-\frac{\pi}{6}i}(-\xi)^{\frac{1}{3}}}i\xi|\xi|\hat{\bar{A}}_{k+1}(\xi),\quad \xi<0.
				\end{aligned}\right.
			\end{aligned}
		\end{equation}
		This shows that $(\bar{\omega},\bar{A})\in\cX_{\alpha,M}$.
		Moreover, it is easy to verify that $(\omega_0+\bar{\omega},A_0+\bar{A})\in \cX_{\alpha,M}\setminus\{0\}$ due to $F\not\equiv0$, 
		which implies that the space $\cX_{\alpha,M}$ is non-trivial. 
		
		Let $(\omega_i,A_i)\ (i=1,2)\in \cX_{\alpha,M}$ be two solutions to the perturbation system \eqref{Perturbation system0}. 
		Set $\tilde{\omega} = \omega_1 - \omega_2,\ \tilde{A} = A_1-A_2$, which solve
		\begin{equation}\nonumber
			\left\{\begin{aligned}
				&y\p_x\tilde{\omega} - \p_y^2\tilde{\omega} = -u_1\p_x\tilde{\omega} - v_1\p_y\tilde{\omega} - \tilde{u}\p_x\omega_2 - \tilde{v}\p_y\omega_2\triangleq \tilde{h},\\
				&\tilde{u} = I_y[\tilde{\omega}],\quad \tilde{v} = -\p_xI_y[\tilde{u}],\\
				&\p_y\tilde{\omega}\big|_{y=0} = \p_x|\p_x|\tilde{A},\quad \tilde{\omega}\big|_{y=M} = \frac{Ai(M\p_x^{\frac{1}{3}})}{Ai'(0)}\p_x^{\frac{2}{3}}|\p_x|\tilde{A},\quad I_M[\tilde{\omega}] = \tilde{A}.
			\end{aligned}\right.
		\end{equation}
		As in the proof of Proposition \ref{prop interation}, we decompose $\tilde{\omega} = \tilde{\omega}_e + \tilde{\omega}_b$ with
		\begin{equation}\nonumber
			\left\{\begin{aligned}
				&y\p_x\tilde{\omega}_e - \p_y^2\tilde{\omega}_e = \tilde{h},\\
				&\p_y\tilde{\omega}_e\big|_{y=0} = 0,\quad \tilde{\omega}_e\big|_{y=M} = 0,
			\end{aligned}\right.
		\end{equation}
		and 
		\begin{equation}\nonumber
			\left\{\begin{aligned}
				&y\p_x\tilde{\omega}_b - \p_y^2\tilde{\omega}_b = 0,\\
				&\p_y\tilde{\omega}_b\big|_{y=0} = \p_x|\p_x|\tilde{A},\quad I_M[\tilde{\omega}_e+\tilde{\omega}_b] = \tilde{A},\\
				&\tilde{\omega}\big|_{y=M} = \frac{Ai(M\p_x^{\frac{1}{3}})}{Ai'(0)}\p_x^{\frac{2}{3}}|\p_x|\tilde{A}.
			\end{aligned}\right.
		\end{equation}
		By Proposition \ref{lem for nonlinear terms}, we have
		\begin{equation}\label{Estimate of baromegae}
			\begin{aligned}
				\norm{\tilde{\omega}_e}_{X_0\cap X_{\alpha}}\le&\ CM^2\norm{\tilde{h}}_{L^2(\Omega_M)} + C\norm{\tilde{h}}_{\dot{H}_x^{\frac{\alpha}{2}}L_y^2}.
			\end{aligned}
		\end{equation}
		
		By Proposition \ref{lem for boundary data}, we get
		\begin{equation}\label{Formula of hatomegab}
			\begin{aligned}
				\hat{\tilde{\omega}}_b(\xi,y) = C(\xi)\big(\sqrt{3}Ai((i\xi)^{\frac{1}{3}}y) + Bi((i\xi)^{\frac{1}{3}}y)\big)+\frac{Ai((i\xi)^{\frac{1}{3}}y)}{Ai'(0)(i\xi)^{\frac{1}{3}}}i\xi|\xi|\hat{\tilde{A}}.
			\end{aligned}
		\end{equation}
		In view of $\hat{\tilde{\omega}}_e\big|_{y=M} = 0$, we have 
		\begin{equation}\nonumber
			\hat{\tilde{\omega}}_b\big|_{y=M} = \hat{\tilde{\omega}}\big|_{y=M} = \frac{Ai((i\xi)^{\frac{1}{3}}M)}{Ai'(0)}(i\xi)^{\frac{2}{3}}|\xi|\hat{\tilde{A}},
		\end{equation}
		which along with \eqref{Formula of hatomegab} gives
		$$
		C(\xi) = 0.
		$$
		Thus, we have
		\begin{equation}\nonumber
			\begin{aligned}
				\hat{\tilde{\omega}}_b(\xi,y) = \frac{Ai((i\xi)^{\frac{1}{3}}y)}{Ai'(0)(i\xi)^{\frac{1}{3}}}i\xi|\xi|\hat{\tilde{A}}.
			\end{aligned}
		\end{equation}
		This formula is identical to \eqref{Explicit form of wb}. We then get by Proposition \ref{lem for boundary data} that
		\begin{equation}\label{Estimate of baromegab}
			\begin{aligned}
				\norm{\tilde{\omega}_b}_{X_0\cap X_{\alpha}}\le C\norm{|\p_x|^{\frac{5}{6}}\tilde{A}}_{H^{\frac{4}{3}+\frac{\alpha}{2}}}\le C\norm{\tilde{A}}_{X_{\alpha,\infty}}.
			\end{aligned}
		\end{equation}
		
		Thanks to $I_M[\tilde{\omega}_e + \tilde{\omega}_b] = \tilde{A}$, we have
		\begin{equation}
			\tilde{A} = m(-i\p_x)^{-1}I_M[\tilde{\omega}_e].\nonumber
		\end{equation}
		It follows from Lemma \ref{prop for IMf} that
		\begin{equation}\label{Estimate of barA}
			\norm{\tilde{A}}_{X_{\alpha,\infty}}\le C\norm{\tilde{\omega}_e}_{X_0\cap X_{\alpha}}.
		\end{equation}
		By Remark \ref{rem for forced terms}, we obtain
		\begin{equation}\label{Estimate of barh}
			\begin{aligned}
				&\norm{\tilde{h}}_{L^2}\le\ CM^{\frac{1+3\e}{2}}\big(\norm{\omega_1}_{X_0\cap X_{\alpha}} + \norm{\tilde{A}}_{ X_{\alpha,\infty}} + \norm{F}_{H^2(\R)}\big)\norm{\tilde{\omega}}_{X_0\cap X_{\alpha}}\\
				&\qquad\qquad\quad+ CM^{\frac{1+3\e}{2}}\big(\norm{\tilde{\omega}}_{X_0\cap X_{\alpha}} + \norm{\tilde{A}}_{ X_{\alpha,\infty}}\big)\norm{\omega_2}_{X_0\cap X_{\alpha}},\\
				&\norm{\tilde{h}}_{\dot{H}_x^{\frac{\alpha}{2}}L_y^2}\le CM\big(\norm{\omega_1}_{X_0\cap X_{\alpha}} + \norm{\tilde{A}}_{ X_{\alpha,\infty}} + \norm{F}_{H^2(\R)}\big)\norm{\tilde{\omega}}_{X_0\cap X_{\alpha}}\\
				&\qquad\qquad\quad + CM\big(\norm{\tilde{\omega}}_{X_0\cap X_{\alpha}} + \norm{\tilde{A}}_{ X_{\alpha,\infty}}\big)\norm{\omega_2}_{X_0\cap X_{\alpha}}.
			\end{aligned}
		\end{equation}
		On the other hand, we know that
		\begin{equation}\label{Linf estimate of barA}
			\begin{aligned}
				\norm{\tilde{A}}_{L^{\infty}(\R)}\le \norm{\tilde{\omega}}_{L_x^{\infty}L_y^1}\le \norm{\tilde{\omega}}_{X_0\cap X_{\alpha}}.
			\end{aligned}
		\end{equation}
		
		Summing up \eqref{Estimate of baromegae}, \eqref{Estimate of baromegab}, \eqref{Estimate of barA}, \eqref{Estimate of barh}  and \eqref{Linf estimate of barA}, we obtain
		\begin{equation}\nonumber
			\begin{aligned}
				&\norm{(\tilde{\omega},\tilde{A})}_{\cX_{\alpha,M}}\\
				&\le CM^{\frac{5+3\e}{2}}\norm{(\tilde{\omega},\tilde{A})}_{\cX_{\alpha,M}}\big(\norm{\omega_1}_{X_0\cap X_{\alpha}} + \norm{\omega_2}_{X_0\cap X_{\alpha}} + \norm{\tilde{A}}_{ X_{\alpha,\infty}} + \norm{F}_{H^2(\R)}\big)\\
				&\le CM^{\frac{5+3\e}{2}}\norm{(\tilde{\omega},\tilde{A})}_{\cX_{\alpha,M}}\big(\norm{(\omega_1,\tilde{A})}_{\cX_{\alpha,M}} + \norm{(\omega_2,A_2)}_{\cX_{\alpha,M}} + \norm{F}_{H^2(\R)}\big).
			\end{aligned}
		\end{equation}
		This implies the uniqueness of solution $(\omega,A)$ in $\cX_{\alpha,M}$ to \eqref{Perturbation system0} when $\norm{(\omega,A)}_{\cX_{\alpha,M}} + \norm{F}_{H^2(\R)}\le \frac{\delta_0}{M^{\frac{5+3\e}{2}}}$ with $\delta_0=\delta_0(\alpha,\e)\ll1$.
	\end{proof}

	\appendix	
	
	\section{The physical derivation of triple-deck equation}

	This appendix provides a brief overview of the triple-deck formalism, based on the physical model of an incompressible boundary layer with a roughness element; for further details, see \cite{Smith1973}. The roughness is located  at a distance $L$ downstream from the leading edge. Taking $L$ as the reference length, the Reynolds number is defined as  $ \text{Re}=\frac{U_\infty L}{\nu}$, with $U_\infty$ and $\nu$ being the velocity of the incoming flow and the kinematic viscosity, respectively. For convenience, we introduce a small parameter $\ep=\text{Re}^{-\frac{1}{8}}\ll 1$. The system is described in the Cartesian coordinate $(X,Y)$, with its origin at the leading edge. The roughness height is assumed to be $\cO(\ep^5)$ or smaller, and its streamwise length scale is $\cO(\ep^3)$. Consequently, the roughness-induced  mean-flow distortion exhibits three distinct layers in the transverse direction, spanning in   the $\cO(\ep^3)$ vicinity of the roughness (near $X=1$). The three distinct layers are summarized as follows.\smallskip
	
	\par I. Viscous sublayer: Vertical scale \(\cO(\ep^5)\), streamwise scale \(\cO(\ep^3)\).
	\par II. Main layer: Vertical scale \(\cO(\ep^4)\), streamwise scale \(\cO(\ep^3)\).
	\par III. Upper layer: Vertical scale \(\cO(\ep^3)\), streamwise scale \(\cO(\ep^3)\).
	
	\medskip
	
	These scales arise from matched asymptotic expansions, connecting the triple-deck theory to classical boundary layer theory in the outer flow region. For clarity, the spatial arrangement of layers is illustrated in Figure \ref{fig3}, showing the viscous sublayer I adjacent to the wall and the upper layer III in the potential flow region.
	
	\begin{figure}[h]
		\centering
		\includegraphics[scale=0.27]{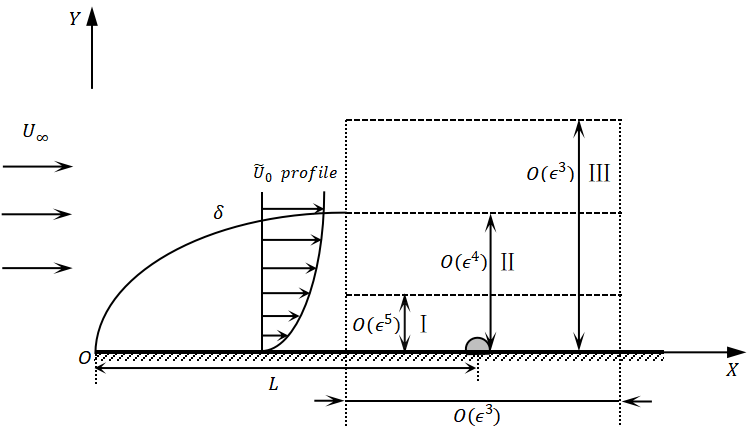}
		\caption{Demonstration of the triple-deck structure (not to scale): I   Viscous Sublayer; II  Main Layer;  III Upper Layer; $\delta$ Thickness of Prandtl layer; $\tilde{U}_0$ Streamwise velocity.}
		\label{fig3}
	\end{figure}
	
	It is well-known that the 2D steady incompressible Navier-Stokes equations take the following dimensionless form:
	\begin{equation}\label{Incompressible NS}
		\left\{\begin{aligned}
			&\fu_{NS}\cdot\nabla\fu_{NS} - \ep^8\Delta\fu_{NS} + \nabla p_{NS} = 0,\\
			&\nabla\cdot\fu_{NS} = 0.
		\end{aligned}\right.
	\end{equation}
	In the main layer (layer II), we introduce a pair of local  coordinates,
	$$
	x_* = \ep^{-3}(X-1)=\cO(1), \quad  \tilde{y} = \ep^{-4}Y = \cO(1).
	$$
	The flow field in this layer can be expressed as a sum of the Blasius base flow $\fu=(\tilde U_0(y),0)+\cO(\ep^4)$ and a mean-flow distortion,
	\begin{equation}\label{Main expansion}
		\begin{aligned}
			u_{NS}(X,Y;\ep) =&\ \tilde{U}_0(\tilde{y}) + \ep\tilde{u}(x_*,\tilde{y}) + o(\ep),\\
			v_{NS}(X,Y;\ep) =&\ \ep^2\tilde{v}(x_*,\tilde{y}) + o(\ep^2),\\
			p_{NS}(X,Y;\ep) =&\ \ep^2\tilde{p}(x_*,\tilde{y}) + o(\ep^2).
		\end{aligned}
	\end{equation}
	Substituting the expansions \eqref{Main expansion} into \eqref{Incompressible NS} and collecting the leading-order terms, we derive the following system
	\begin{equation}
		\left\{\begin{aligned}
			\tilde{U}_0\p_{x_*}\tilde{u} + \tilde{U}_0'\tilde{v} =&\ 0,\\
			\p_{\tilde{y}}\tilde{p} =&\ 0,\\
			\p_{x_*}\tilde{u} + \p_{\tilde{y}}\tilde{v} =&\ 0.
		\end{aligned}\right.
	\end{equation}
	The solutions to this system are
	\begin{equation}\label{Main layer sol1}
		\begin{aligned}
			\tilde{u} = \lambda^{-1} A_*(x_*)\tilde{U}_0'(\tilde{y}),\quad
			\tilde{v} = -\lambda^{-1}A_*'(x_*)\tilde{U}_0(\tilde{y}),\quad
			\tilde{p} = \tilde{p}(x_*).
		\end{aligned}
	\end{equation}
	Here $\lambda:=\tilde{U}_0'(0)>0$ and $A_*$ is an unknown function (typically referred to as a displacement of velocity) satisfying $A_*(-\infty)=0$.
	
	As $\tilde{y}\to\infty$, we find that $\tilde{v}$ approaches $-\lambda^{-1}A_*'(x_*)$, which does not satisfy   the far-field attenuation condition, $\lim\limits_{Y\uparrow\infty}v_{NS}(X,Y)=0$. This indicates the emergence of an upper layer (layer III in Figure \ref{fig3}).  In this upper layer, we introduce the local coordinate
	$$
	\bar{y} = \ep^{-3}Y = \cO(1).
	$$
	Matching with the upper limit of the main-layer solutions, we can express the flow field as
	\begin{equation}\label{Upper expansion}
		\begin{aligned}
			u_{NS}(X,Y;\ep) =&\ 1 + \ep^2\bar{u}(x_*,\bar{y}) + o(\ep^2),\\
			v_{NS}(X,Y;\ep) =&\ \ep^2\bar{v}(x_*,\bar{y}) + o(\ep^2),\\
			p_{NS}(X,Y;\ep) =&\ \ep^2\bar{p}(x_*,\bar{y}) + o(\ep^2).
		\end{aligned}
	\end{equation}
	Substituting the expansions \eqref{Upper expansion} into \eqref{Incompressible NS} and retaining the leading-order terms, we derive
	\begin{equation}\label{C-R for barp barv}
		\left\{\begin{aligned}
			\p_{x_*}\bar{u} + \p_{x_*}\bar{p} =&\ 0,\\
			\p_{x_*}\bar{v} + \p_{\bar{y}}\bar{p} =&\ 0,\\
			\p_{x_*}\bar{u} + \p_{\bar{y}}\bar{v} =&\ 0.
		\end{aligned}\right.
	\end{equation}
	The far-field attenuation condition indicates
	\begin{equation}
		\begin{aligned}
			\lim_{\sqrt{x_*^2 + \bar{y}^2}\to\infty}(\bar u,\bar v, \bar p) = \mathbf{0}.
		\end{aligned}
	\end{equation}
	Matching with the main layer,    we derive
	\begin{equation}\label{BC2}
		\begin{aligned}
			\bar{v}(x_*, 0) = \lim_{\tilde{y}\uparrow\infty}\tilde{v} = -\lambda^{-1}A_*'(x_*),\quad \bar{p}(x_*, 0) = \tilde{p}(x_*).
		\end{aligned}
	\end{equation}
	Since $(\bar{p}, \bar{v})$ satisfies the Cauchy-Riemann equations (cf. \eqref{C-R for barp barv}), let us define $f(z) = \bar{p}(x_*\bar{y}) + i\bar{v}(x_*,\bar{y})$ with $z = x_* + i\bar{y}$. Then, we obtain
	\begin{equation}\label{Sol0}
		\begin{aligned}
			f(z) =&\ \frac{1}{2\pi i}\int_{\R}\frac{f(\zeta) - \overline{f(\zeta)}}{\zeta - z}d\zeta =\ -\frac{1}{\pi\lambda}\int_{\R}\frac{A_*'(\zeta)}{\zeta - z}d\zeta,\quad \forall\ \text{Im }z>0.
		\end{aligned}
	\end{equation}
	This implies that
	\begin{equation}\label{BC3}
		\tilde p(x_*)= \bar{p}(x_*,0) = \lim_{\text{Im}(z) \to 0^+}\text{Re}\ f(z) = -\frac{1}{\pi\lambda}\text{P.V.}\int_{\R}\frac{A_*'(\zeta)}{\zeta - x_*} d\zeta.
	\end{equation}
	
	Now let us consider the lower boundary of the main-layer solution.
	As $\tilde{y}\to0^+$, we find that $\tilde{u}$ approaches $A_*(x_*)$, disagree with the no-slip condition. Thus, we need to consider the viscous sublayer (layer I in Figure \ref{fig3}). 
	For convenience, we introduce
	$$
	y_* = \ep^{-5}Y = \cO(1).
	$$
	Analysis of the Blasius equation leads to   $\tilde{U}_0\to \lambda \tilde{y} +\cO(\tilde{y}^4)$ as $\tilde{y}\to 0^+$. Thus, the main-layer solution behaves like $u_{NS}=\bar U_0+\epsilon\tilde u+o(\ep)\to \ep(\lambda y_*+A_*)+o(\ep)$, which determines the magnitude of the sublayer-layer solution. From the balance of the NS equations, we can expand the solution as,
	\begin{equation}\label{Viscous expansion}
		\begin{aligned}
			u_{NS}(X,Y;\ep) =&\ \ep u_*(x_*,y_*) + o(\ep),\\
			v_{NS}(X,Y;\ep) =&\ \ep^3 v_*(x_*,y_*) + o(\ep^3),\\
			p_{NS}(X,Y;\ep) =&\ \ep^2P_*(x_*) + o(\ep^2).
		\end{aligned}
	\end{equation}
	Here, we have imposed the fact that, to leading order, the pressure remains uniform across the thin sublayer. 
	Substituting \eqref{Viscous expansion} into \eqref{Incompressible NS} and collecting the leading-order terms, we derive the following system:
	\begin{equation}\label{Eq in viscous layer}
		\left\{\begin{aligned}
			u_*\p_{x_*}u_* + v_*\p_{y_*}u_* + \frac{dP_*}{dx_*} - \p_{y_*}^2u_* =&\ 0,\\
			\p_{x_*}u_* + \p_{y_*}v_* =&\ 0.
		\end{aligned}\right.
	\end{equation}
	The velocity field is subject to the no-slip, non-penetration conditions,
	\begin{equation}\label{BC0}
		[u_*,v_*]\big|_{y_* = F(x_*)} = 0.
	\end{equation}
	In the upstream limit, the velocity approaches the unperturbed Blasius profile,
	\begin{equation}\label{BC1}
		(u_* - \lambda y_*)\big|_{x_*\downarrow-\infty} = 0,\quad P_*\big|_{x_*\downarrow-\infty} = 0.
	\end{equation}
	
	The upper boundary condition of the sublayer equations should be obtained by matching with the main layer. Now, we consider the overlapping region connecting the sublayer and main layer, where $y_* = \cO(\ep^{-\alpha_0})$ with $\alpha_0\in (0,1)$.  Then for any $\beta_0>0$,  we have
	\begin{equation}\label{Viscous-main compatible}
		\begin{aligned}
			u_{NS}=&\ \lambda\tilde{y} + \ep\tilde{u}(x_*,0) + o(\ep+\tilde{y}^{4-\beta_0}) = \lambda\tilde{y} + \ep A_*(x_*) + o(\ep + \tilde{y}^{4-\beta_0})\\
			=&\ \ep\big(\lambda y_* + A_*(x_*)\big) + o\big(\ep + \ep^{(4-\beta_0)(1-\alpha_0)}\big).
		\end{aligned}
	\end{equation}
	To ensure the second part to be negligible, we require $(4-\beta_0)(1-\alpha_0)\geq 1$, namely, $\alpha_0\in(0,3/4)$ and $\beta_0\in(0,4)$. Likewise, we can derive that the leading-order pressure in the sublayer and that in the main layer are equal. Subsequently, we obtain
	\begin{equation}\label{B.C. for viscous layer}
		\big(u_*(x_*,y_*)-\lambda y_*\big)\big|_{y_*\uparrow\infty} = A_*(x_*),\quad \tilde{p}(x_*) = P_*(x_*).
	\end{equation}
	From \eqref{Eq in viscous layer}-\eqref{BC1} and \eqref{B.C. for viscous layer}, we can derive the classical complete system in viscous sublayer as follows
	\begin{equation}\label{Viscous sublayer system0}
		\begin{aligned}
			\left\{\begin{aligned}
				&u_*\p_{x_*}u_* + v_*\p_{y_*}u_* + \frac{dP_*}{dx_*} - \p_{y_*}^2u_* = 0,\quad  (x_*,y_*)\in\Omega_{\infty,F},\\
				&\p_{x_*}u_* + \p_{y_*}v_* = 0,\quad  (x_*,y_*)\in\Omega_{\infty,F},\\
				&[u_*,v_*]\big|_{y_*=F(x_*)} = 0,\quad (u_*-\lambda y_*)\big|_{y_*\uparrow\infty} = A_*(x_*),\quad x\in\R,\\
				&(u_*-\lambda y_*)\big|_{x_*\downarrow-\infty} = 0,\quad y_*\in\R_+,\\
				&P_*(x_*) = -\frac{1}{\pi\lambda}\text{P.V.}\int_{\R}\frac{A_*'(\zeta)}{\zeta-x_*}d\zeta,\quad x\in\R.
			\end{aligned}\right.
		\end{aligned}
	\end{equation}
	Here $\Omega_{\infty,F} = \big\{(x_*,y_*)\big|\ x_*\in\R,\ y_*>F(x_*)\big\}$ and $\lim_{x_*\downarrow-\infty} F(x_*)=0$.
	
	Indeed, the above nonlinear system is usually solved using numerical approach.   In light of \eqref{Viscous-main compatible}, it is permissible to modify the upper boundary of the domain $\Omega_{\infty,F}$ to be a curve $\gamma=\{(x_*,y_*)|\ y_* = f(x_*)\in \cO(\ep^{-\alpha_0}),\ x_*\in\R\}$. Consequently, the solution $\phi_*=(u_*,v_*,P_*)$  should be denoted by $\phi_{*,\ep}$ to include the impact of the upper boundary. It is clear that $\phi_{*,\ep}$ also satisfy the system \eqref{Viscous sublayer system0}, except for the upper boundary condition \eqref{B.C. for viscous layer}, which is converted to
	\begin{equation}\label{B.C. for viscous layer bounded}
		( u_{*,\ep}(x_*,y_*)-\lambda y_*)\big|_{y_*=f(x_*)} = A_*(x_*).
	\end{equation}
	When there is no ambiguity, we still denote $\phi_{*,\ep}$ as $\phi_*$ for simplicity.

	There is a useful operation named the Prandtl transformation for solving the system \eqref{Viscous sublayer system0}, which takes the following form: 
	\begin{equation}
		\begin{aligned}
			&y = y_* - F(x_*),\ v_*=V+u_*F'(x_*).
		\end{aligned}
	\end{equation}
	This transformation merely shifts down the domain by a distance $F(x_*)$ in the $y$-direction, with the governing equations being unchanged. It is preferable to take $f(x_*) = \e^{-\alpha_0} + F(x_*)$,
	then, the boundary conditions in the transverse direction are changed to
	\begin{equation}
		(u_*,V)\Big|_{y=0}=0,\quad (u_*-\lambda y)\Big|_{y=\ep^{-\alpha_0}}=A_*+\lambda F.
	\end{equation}  
	Typically, we take $\lambda=1$ in this paper.

	\section*{Acknowledgments} 
	
	M. Dong is partially supported by the NSFC under Grant No. 12588201 and the CAS Project for Young Scientists in Basic Research under Grant No. YSBR-087. C. Wang is partially supported by the NSFC under Grant No. 12471189. Z. Zhang is partially supported by the NSFC under Grant No. 12288101.
	
	\bigskip
	{\bf Data availability} The manuscript has no associated data.
	
	\bigskip
	
	{\bf Declarations}
	
	\bigskip
	
	{\bf Conflict of interest} The authors state that there is no conflict of interest.

\end{document}